\documentclass[reqno]{amsart}

\usepackage [mathscr]{eucal}
\usepackage{amsmath,amsthm,amssymb,amscd}
\usepackage{amsrefs}
\usepackage{epsf}
\usepackage{amsfonts}
\usepackage{dsfont}
\usepackage{latexsym}
\usepackage{mathrsfs}
\usepackage{layout}
\usepackage{bm}
\usepackage{verbatim}

\newtheorem{theorem}{Theorem}[section]
\newtheorem{lemma}[theorem]{Lemma}
\newtheorem{proposition}[theorem]{Proposition}
\newtheorem{corollary}[theorem]{Corollary}

\newtheorem{remark}[theorem]{Remark}
\newtheorem{definition}[theorem]{Definition}

\numberwithin{equation}{section}

\def\Xint#1{\mathchoice
{\XXint\displaystyle\textstyle{#1}}%
{\XXint\textstyle\scriptstyle{#1}}%
{\XXint\scriptstyle\scriptscriptstyle{#1}}%
{\XXint\scriptscriptstyle\scriptscriptstyle{#1}}%
\!\int}
\def\XXint#1#2#3{{\setbox0=\hbox{$#1{#2#3}{\int}$ }
\vcenter{\hbox{$#2#3$ }}\kern-.6\wd0}}

\def\dint{\Xint-}

\DeclareMathOperator *{\essosc}{ess\ osc}
\DeclareMathOperator *{\osc}{osc}
\DeclareMathOperator *{\esssup}{ess\ sup}
\DeclareMathOperator *{\essinf}{ess\ inf}
\DeclareMathOperator *{\di}{div} 
\DeclareMathOperator *{\meas}{meas}
\DeclareMathOperator *{\dist}{dist}
\DeclareMathOperator *{\data}{data}
\DeclareMathOperator *{\diam}{diam}
\DeclareMathOperator *{\loc}{loc}
\DeclareMathOperator *{\Lip}{Lip}
\DeclareMathOperator *{\Tr}{Tr}
\DeclareMathOperator *{\BMO}{BMO}
\DeclareMathOperator *{\Proj}{Proj}
\DeclareMathOperator *{\BBR}{\mathbb{R}}
\DeclareMathOperator *{\BBC}{\mathbb{C}}

\newcommand{\cJ}{{\mathcal J}}

\begin{document}
\title[Elliptic PDEs with complex coefficients]{
Boundary value problems for second order elliptic operators with complex coefficients}

\author{Martin Dindo\v{s}}
\address{School of Mathematics, \\
         The University of Edinburgh and Maxwell Institute of Mathematical Sciences, UK}
\email{M.Dindos@ed.ac.uk}

\author{Jill Pipher}
\address{Department of Mathematics, \\ 
	Brown University, USA}
\email{jill\_pipher@brown.edu}

\begin{abstract}
The theory of second order complex coefficient operators of the form $\mathcal{L}=\mbox{div} A(x)\nabla$
has recently been developed under the assumption of $p$-ellipticity. In particular,  if the matrix $A$ is $p$-elliptic, 
the solutions $u$ to $\mathcal{L}u = 0$ will satisfy a higher integrability, even though they may not be continuous in the interior. 
Moreover, these solutions have the property that $|u|^{p/2-1}u \in W^{1,2}_{loc}$.
These properties of solutions were used by Dindo\v{s}-Pipher to solve the $L^p$ Dirichlet problem for $p$-elliptic operators whose
coefficients satisfy a further regularity condition, a Carleson measure condition that has often 
appeared in the literature in the study of real, elliptic divergence form operators.
This paper contains two main results.  First, we establish solvability of the Regularity boundary value problem for 
this class of operators, in the same range as that of the Dirichlet problem. The Regularity problem, even
in the real elliptic setting, is more delicate than the Dirichlet problem because it requires estimates on derivatives of 
solutions.
Second, the Regularity results allow us to extend the previously 
established range of $L^p$ solvability of the Dirichlet problem using a theorem due to Z. Shen for general bounded
sublinear operators. 
\end{abstract}

\maketitle

\section{Introduction}\label{S:Intro}

The theory of elliptic boundary value problems under minimal smoothness assumptions on the boundary or the coefficients has been well-studied in the real-valued setting and there
is a rich literature of results and applications. 
By contrast, the literature in the complex valued setting is much more limited. Some important milestones in the study of complex coefficient operators exist: notable is the resolution of the Kato problem, which can be formulated as a ``Regularity" boundary value problem for operators that satisfy very specific constraints in structure (\cite{AHLMT}).  The challenge in this theory is that solutions to complex coefficient elliptic operators are not necessarily 
continuous, nor do they satisfy even a weak maximum principle, which is typically the starting point for the study of boundary value problems. Some of the results for complex coefficient operators have been proven under the assumption of interior H\"older regularity 
(the De Giorgi-Nash-Moser theory), yet it is not clear how this assumption can be correlated with quantitative verifiable assumptions on the operators.

In this paper we continue the investigation of solvability of boundary value problems for complex valued second order divergence form elliptic operators under a structural algebraic assumption on the matrix known as 
{\it $p$-ellipticity}. This structural assumption was introduced independently in \cite{DPcplx} and \cite{CD}, and is a quantitative strengthening of a condition related to $L^p$-contractivity of elliptic operators that was discovered by Cialdea and Maz'ya (\cite{CM1}). When the coefficients of the operator are real, or when $p=2$, the $p$-ellipticity condition is equivalent to the familiar uniform ellipticity condition.

In \cite{DPcplx} we used the $p$-ellipticity condition to establish a limited interior regularity for solutions to these complex coefficient second order divergence form operators. 
We think of this as a weak substitute for the De Giorgi-Nash-Moser regularity of real valued operators and, in fact, we used a variant of Moser's iteration argument to prove it.
Specifically, we considered there operators of the form
$\mathcal{L}=\mbox{div} A(x)\nabla +B(x)\cdot\nabla$, where the matrix $A$ is $p$-elliptic and $B$ satisfies a natural minimal scaling condition. 
This limited regularity theory allowed us to address the solvability of the $L^p$ Dirichlet problem for a collection of operators with complex coefficients whose matrices are in canonical form, as defined below. (\cite{DPcplx} contains a discussion of how to put an operator with lower order terms in canonical form.)

This results of this paper concern the aforementioned Regularity problem, in which the boundary data is prescribed in the Sobolev space of functions whose tangential derivatives belong
to some $L^p$ space. In analogy with the Dirichlet problem, where one expects to show classical convergence of a solution nontangentially to its boundary data in $L^p$ through the control
of a nontangental maximal function, in this problem one expects to prove nontangential estimates for the 
gradients of the solution in terms of the derivatives of the data on the boundary.  The formulation of these estimates must take into account the fact these solutions and their derivatives do not have pointwise values, but are merely measurable functions in certain Lebesgue spaces.

We now discuss the class of elliptic operators for which Dirichlet and Regularity problems are considered. In \cite{KP01}, a class of real valued second order operators (with drift terms like those defined below) was introduced, and the elliptic measure associated to such operators was shown to belong to the $A_\infty$ class with respect to surface measure on the boundary. This implies that the Dirichlet problem for these operators is solvable with data in $L^p$ for some possibly large value of $p$. The study of this class of operators was motivated by a question of Dahlberg, which in turn was inspired by the fact that these operators arose naturally from a change of variables mapping from Lipschitz into flat domains. Specifically, the coefficients of the matrix $A$ was
assumed to satisfy a Carleson measure.  Examples showed that $A_\infty$ was the optimal result in this regime. Later, a slight strengthening of the Carleson measure condition was shown in \cite{DPP} to imply solvability of the Dirichlet problem for the full range $1 < p < \infty$. We refer to this condition as the 
``small" Carleson condition, defined in Section \ref{S-not}. 

In \cite{DPR}, this Regularity problem was solved for equations of the form $\mathcal{L}=\mbox{div} A(x)\nabla$, with $A$ {\it real and elliptic}, satisfying this small Carleson condition. There are open questions even for operators with real coefficients that satisfy the Carleson condition of
\cite{KP01}, such as solvability of the Regularity problem in $L^p$ for $p$ near 1.

\smallskip
The first main result of this paper is the solvability of the Regularity problem for boundary data $\nabla_Tf\in L^p$, under the assumption that the matrix $A$ is $p$-elliptic and satisfies small Carleson condition.  

\begin{theorem}\label{S3:T0}  Let $1<p<\infty$, and let $\Omega$ be the upper half-space ${\mathbb R}^n_+=\{(x_0,x'):\,x_0>0\mbox{ and } x'\in{\mathbb R}^{n-1}\}$. Consider the operator 
$$ \mathcal Lu = \partial_{i}\left(A^0_{ij}(x)\partial_{j}u\right) 
$$
and assume that $\mathcal L$ can be re-written as
\begin{equation}
 \mathcal Lu = \partial_{i}\left(A_{ij}(x)\partial_{j}u\right) +B_i\partial_iu\label{eq-oper-mod}
\end{equation}
where the matrix $A$ is $p$-elliptic with constants $\lambda_p,\Lambda$, $A_{00}=1$ and $\mathscr{I}m\,A_{0j}=0$ for all $1\leq j \leq n-1$.
Assume also that
\begin{equation}\label{Car_hatAA}
d{\mu}(x)=\sup_{B_{\delta(x)/2}(x)}(|\nabla{A}|^{2}+|B|^2) \delta(x)\,dx
\end{equation}
is a Carleson measure in $\Omega$. 

Then there exist $K=K(\lambda_p,\Lambda,n,p)>0$ and $C(\lambda_p, \Lambda ,n,p)>0$ such that if
\begin{equation}\label{Small-Cond}
\|\mu\|_{\mathcal C} < K
\end{equation}
then the $L^p$ Regularity problem  
 
 \begin{equation}\label{E:R2}
\begin{cases}
\,\,{\mathcal L}u=0 
& \text{in } \Omega,
\\[4pt]
\quad u=f & \text{ for $\sigma$-a.e. }\,x\in\partial\Omega, 
\\[4pt]
\tilde{N}_{p,a}(\nabla u) \in L^{p}(\partial \Omega), &
\end{cases}
\end{equation}
is solvable and the estimate
\begin{equation}\label{Main-Est}
\|\tilde{N}_{p,a} (\nabla u)\|_{L^{p}(\partial \Omega)}\leq C\|\nabla_T f\|_{L^{p}(\partial \Omega;{\BBC})}
\end{equation}
holds for all energy solutions $u$ with datum $f$.
\end{theorem}

The second main theorem of the paper extends the range of 
solvability of $\mathcal Lu=0$ with $L^{p}$ Dirichlet boundary data for variable coefficient complex coefficient operators satisfying these Carleson conditions on coefficients.  In the paper \cite{DPcplx} we have considered the solvability in the range $p\in (p_0,p_0')$ where
\begin{equation}\label{eqp0}
p_0=\inf\{p>1:\, \mbox{the matrix $A$ is $p$-elliptic}\}.
\end{equation}

Thanks to the solvability of the Regularity problem (Theorem \ref{S3:T0}) we are now able to use the technique of Z. Shen (\cite{Sh1}, \cite{Sh2}) and
extend the previously established range of solvability of the Dirichlet problem to a larger interval $p\in (p_0,p_0'\frac{n-1}{n-1-p_0'})$. In particular,  
when $n=2,3$ or when $p_0'>n-1$, the range of solvability is extended to all $p\in (p_0,\infty).$

\begin{theorem}\label{S3:T1}  Consider the operator 
$$ \mathcal Lu = \partial_{i}\left(A^0_{ij}(x)\partial_{j}u\right) 
$$
in the domain $\Omega={\mathbb R}^n_+=\{(x_0,x'):\,x_0>0\mbox{ and } x'\in{\mathbb R}^{n-1}\}$. Asume again that $\mathcal L$ can be rewritten as \eqref{eq-oper-mod} and
let $p_0$ be defined as in \eqref{eqp0} and let $p_{max}=\infty$ when $p_0'\ge n-1$, 
$$p_{\max}=\frac{p_0'(n-1)}{n-1-p_0'},$$
otherwise. Finally consider any  $p_0<p<p_{max}$.

Assume further that the matrix $A$, satisfies $A_{00}=1$, $\mathscr{I}m\,A_{0j}=0$ for all $1\leq j \leq n-1$ and let
\begin{equation}\label{Car_hatAA-2}
d{\mu}(x)=\sup_{B_{\delta(x)/2}(x)}(|\nabla{A}|^{2}+|B|^2) \delta(x)\,dx
\end{equation}
be a Carleson measure in $\Omega$. 

Then there exist $K=K(\lambda_p,\Lambda,n,p)>0$ and $C(\lambda_p, \Lambda , n,p)>0$ such that if
\begin{equation}\label{Small-Cond-2}
\|\mu\|_{\mathcal C} < K
\end{equation}
then the $L^p$-Dirichlet problem  
 
 \begin{equation}\label{E:D2}
\begin{cases}
\,\,{\mathcal L}u=0 
& \text{in } \Omega,
\\[4pt]
\quad u=f & \text{ for $\sigma$-a.e. }\,x\in\partial\Omega, 
\\[4pt]
\tilde{N}_{p,a}(u) \in L^{p}(\partial \Omega), &
\end{cases}
\end{equation}
is solvable and the estimate
\begin{equation}\label{Main-Est2}
\|\tilde{N}_{p,a} (u)\|_{L^{p}(\partial \Omega)}\leq C\|f\|_{L^{p}(\partial \Omega;{\BBC})}
\end{equation}
holds for all energy solutions $u$ with datum $f$.
\end{theorem}

In particular observe that $p_{max}=\infty$ in dimensions $2$ and $3$ and that when $n\ge 4$
$$p_{\max}>\frac{2(n-1)}{n-3}.$$

\begin{remark} We address at the end of section \ref{S-not} how we can rewrite any operator $\mathcal L$ as \eqref{eq-oper-mod} with coefficients $A_{0j}$ real and $A_{00}=1$. We require this particular form of our operator in the main section \ref{S4} of this paper.
\end{remark}

In the statement of these two theorems, we've used some notation that will be defined in subsequent sections. We will also recall there the concept of Carleson measure, discuss the notions of $L^p$ solvability and energy solutions and define $\tilde{N}_{p}$ which is a variant of the nontangential maximal function defined using $L^p$ averages of the solution $u$. 

\begin{remark}
Lemma 2.6 of  \cite{DPcplx} shows that $L^q$ averages of solutions on interior balls are controlled bp $L^q$ averages for $q$ in the range $(p_0,\frac{p_0'n}{n-2})$, extending beyond the range of $p$-ellipticity. Thus one can use the $N_q$ nontangential maximal function for such $q$ in the estimate \eqref{Main-Est2}. The arguments for Theorem \ref{S3:T0} show that, similarly, the gradient $\nabla u$ of solutions to the Regularity problem will be locally $L^q$ integrable for $q$ in the range $(p_0,\frac{p_0'n}{n-2})$. In particular, by Sobolev embedding, solvability of the Regularity problem in the regime $p_0' > n-2$ implies that solutions are H\"older continuous. 
\end{remark}

The paper is organized as follows. In Section \ref{S-not}, we define the concept of $p$-ellipticity, the nontangential maximal function, the 
$p$-adapted square 
function, Carleson measures and the notions of solvability of these various boundary value problems.
In Section \ref{SS:43}, we establish bounds for the nontangential maximal function by the square function.
The estimates for the $p$-adapted square functions are established in Section \ref{S4}. In light of \eqref{eq-partial}, square functions that involve tangential
derivatives are easier to handle and we begin by bounding these. We then show that, essentially, the square function with the full gradient can be
controlled by the square functions of tangential derivatives. 
In Sections \ref{S5} and \ref{S6}, we present the proofs of the two main theorems.

\section{Basic notions and definitions}\label{S-not}

\subsection{$p$-ellipticity}
The concept of $p$-ellipticity
was introduced in \cite{CM}, where the authors investigated the $L^p$-dissipativity of second order divergence complex coefficient operators. Later, Carbonaro and Dragi\v{c}evi\'c \cite{CD} gave an equivalent definition and coined the term ``$p$-ellipticity".  It is this definition that was
most useful for the results of  \cite{DPcplx}. 
To introduce this, we define, for $p>1$, the ${\mathbb R}$-linear map $\cJ_p:{\mathbb C}^n\to {\mathbb C}^n$ by
$$\cJ_p(\alpha+i\beta)=\frac{\alpha}{p}+i\frac{\beta}{p'}$$
where $p'=p/(p-1)$ and $\alpha,\beta\in{\mathbb R}^n$.

\begin{definition}\label{pellipticity} Let $\Omega\subset{\mathbb R}^n$. Let $A:\Omega\to M_n(\mathbb C)$, where $M_n(\mathbb C)$ is the space of $n\times n$ complex valued matrices. We say that $A$ is $p$-elliptic if for a.e. $x\in\Omega$
\begin{equation}\label{pEll}
\mathscr{R}e\,\langle A(x)\xi,\cJ_p\xi\rangle \ge \lambda_p|\xi|^2,\qquad\forall \xi\in{\mathbb C}^n
\end{equation}
for some $\lambda_p>0$ and there exists $\Lambda>0$ such that 
\begin{equation}
|\langle A(x)\xi,\eta \rangle| \le \Lambda |\xi| |\eta|, \qquad\forall \xi, \,\eta\in{\mathbb C}^n.
\end{equation}
\end{definition}

It is now easy to observe that the notion of $2$-ellipticity coincides with the usual ellipticity condition for complex matrices.
As shown in \cite{CD} if $A$ is elliptic, then there exists $\mu(A)>0$ such that $A$ is $p$-elliptic if and only if
$\left|1-\frac2p\right|<\mu(A).$
Also $\mu(A)=\infty$ if and only if $A$ is real valued.

\subsection{Nontangential maximal and square functions}
\label{SS:NTS}

On a domain of the form 
\begin{equation}\label{Omega-111}
\Omega=\{(x_0,x')\in\BBR\times{\BBR}^{n-1}:\, x_0>\phi(x')\},
\end{equation}
where $\phi:\BBR^{n-1}\to\BBR$ is a Lipschitz function with Lipschitz constant given by 
$L:=\|\nabla\phi\|_{L^\infty(\BBR^{n-1})}$, define for each point $x=(x_0,x')\in\Omega$
\begin{equation}\label{PTFCC}
\delta(x):=x_0-\phi(x')\approx\mbox{dist}(x,\partial\Omega).
\end{equation}
In other words, $\delta(x)$ is comparable to the distance of the point $x$ from the boundary of $\Omega$.

\begin{definition}\label{DEF-1}
A cone of aperture $a>0$ is a non-tangential approach region to the point $Q=(x_0,x') \in \partial\Omega$ defined as
\begin{equation}\label{TFC-6}
\Gamma_{a}(Q)=\{(y_0,y')\in\Omega:\,a|x_0-y_0|>|x'-y'|\}.
\end{equation}
\end{definition}

We require $1/a>L$, otherwise the aperture of the cone is too large and might not lie inside $\Omega$.  When $\Omega=\BBR^n_+$ all parameters $a>0$ may be considered.
Sometimes it is necessary to truncate $\Gamma(Q)$ at height $h$, in which case we write
\begin{equation}\label{TRe3}
\Gamma_{a}^{h}(Q):=\Gamma_{a}(Q)\cap\{x\in\Omega:\,\delta(x)\leq h\}.
\end{equation}

\begin{equation}\label{SSS-1}
\|S_{a}(w)\|^{2}_{L^{2}(\partial\Omega)}\approx\int_{\Omega}|\nabla w(x)|^{2}\delta(x)\,dx.
\end{equation}

In [DPP], a ``$p$-adapted" square function was introduced. The usual square function is the $p$-adapted square function when $p=2$. In the following definition, when $p<2$ we use the convention that the expression $|\nabla w(x)|^{2} |w(x)|^{p-2}$ is zero whenever
$\nabla w(x)$ vanishes.

\begin{definition}\label{D:Sp}
For $\Omega \subset \mathbb{R}^{n}$, the $p$-adapted square function of $w:\Omega\to {\mathbb C}$ such that $w|w|^{p/2-1}\in W^{1,2}_{loc}(\Omega; {\BBC})$ at $Q\in\partial\Omega$ relative 
to the cone $\Gamma_{a}(Q)$ is defined by
\begin{equation}\label{yrddp}
S_{p,a}(w)(Q):=\left(\int_{\Gamma_{a}(Q)}|\nabla w(x)|^{2} |w(x)|^{p-2}\delta(x)^{2-n}\,dx\right)^{1/2}
\end{equation}
and, for each $h>0$, its truncated version is given by 
\begin{equation}\label{yrddp.2}
S_{p,a}^{h}(w)(Q):=\left(\int_{\Gamma_{a}^{h}(Q)}|\nabla w(x)|^{2}|w(x)|^{p-2}\delta(x)^{2-n}\,dx\right)^{1/2}.
\end{equation}

We further introduce the following convention. When $w:\Omega\to {\mathbb C}^k$ with component functions
$(w_i)_{1\le i\le k}$ we denote by $S_{p,a}(w)(Q)$ the following sum
\begin{equation}\label{yrddp.4}
S_{p,a}(w)(Q):=\sum_{i=1}^k S_{p,a}(w_i)(Q),
\end{equation}
hence for example if $w=\nabla_T u$ then $S_{p,a}(\nabla_T u)(Q)$ denotes
$$\sum_{i=1}^{n-1}S_{p,a}(\partial_i u)(Q).$$
\end{definition}

It is not immediately clear that the integrals appearing in \eqref{yrddp} are well-defined. However, in
\cite{DPcplx}, it was shown that the expressions of the form $|\nabla w(x)|^{2} |w(x)|^{p-2}$, when $w$ is a solution of $\mathcal Lw=0$, are locally integrable and hence the definition of $S_p(w)$ makes sense for such $p$ whenever $p$-ellipticity holds. This in particular applies with some modifications to $w=\nabla_T u$ on ${\mathbb R}^n_+$. Each component of $w$ solves a PDE ${\mathcal L}(w_k)=\partial_i((\partial_k A_{ij})w_j)-\partial_k(B_i)w_i$. The righthand side of this PDE is good enough for the regularity theory developed in \cite{DPcplx} to apply to this more complicated system of equations as well. \vglue1mm

A simple application of Fubini's theorem gives 
\begin{equation}\label{SSS-2}
\|S_{p,a}(w)\|^{p}_{L^{p}(\partial\Omega)}\approx\int_{\Omega}|\nabla w(x)|^{2}|w(x)|^{p-2}\delta(x)\,dx.
\end{equation}

\begin{definition}\label{D:NT} 
For $\Omega\subset\mathbb{R}^{n}$ as above, and for 
a continuous $w: \Omega \rightarrow \mathbb C$, the nontangential maximal function ($h$-truncated nontangential maximal function) of $u$
 at $Q\in\partial\Omega$ relative to the cone $\Gamma_{a}(Q)$,
is defined by
\begin{equation}\label{SSS-2a}
N_{a}(w)(Q):=\sup_{x\in\Gamma_{a}(Q)}|w(x)|\,\,\text{ and }\,\,
N^h_{a}(w)(Q):=\sup_{x\in\Gamma^h_{a}(Q)}|w(x)|.
\end{equation}
Moreover, we shall also consider a related version  of the above nontangential maximal function.
This is denoted by $\tilde{N}_{p,a}$ and is defined using $L^p$ averages over balls in the domain $\Omega$. 
Specifically, given $w\in L^p_{loc}(\Omega;{\BBC})$ we set
\begin{equation}\label{SSS-3}
\tilde{N}_{p,a}(w)(Q):=\sup_{x\in\Gamma_{a}(Q)}w_p(x)\,\,\text{ and }\,\,
\tilde{N}_{p,a}^{h}(w)(Q):=\sup_{x\in\Gamma_{a}^{h}(Q)}w_p(x)
\end{equation}
for each $Q\in\partial\Omega$ and $h>0$ where, at each $x\in\Omega$, 
\begin{equation}\label{w}
w_p(x):=\left(\dint_{B_{\delta(x)/2}(x)}|w(z)|^{p}\,dz\right)^{1/p}.
\end{equation}
\end{definition}

Above and elsewhere, a barred integral indicates an averaging operation. Observe that, given $w\in L^p_{loc}(\Omega;{\BBC})$, the function $w_p$ 
associated with $w$ as in \eqref{w} is continuous and $\tilde{N}_{p,a}(w)=N_a(w_p)$ 
everywhere on $\partial\Omega$.

The $L^2$-averaged nontangential maximal function was introduced in \cite{KP2} in connection with
the Neuman and regularity problem value problems. In the context of $p$-ellipticity, Proposition 3.5 of \cite{DPcplx} shows that there is no difference between 
$L^2$ averages and $L^p$ averages when $w=u$ solves $\mathcal Lu=0$ and that $\tilde{N}_{p,a}(u)$ and $\tilde{N}_{2,a'}(u)$ are comparable in $L^r$ norms for all $r>0$ and all allowable apertures $a,a'$.

In this paper we shall consider $w=\nabla u$. However, as it turns out a modification of the argument following (2.20) of \cite{DPcplx}
applies in our case: each component $w^k=\partial_k u$ of $w$ solves an equation similar to one considered in \cite{DPcplx}, namely
\begin{equation}\label{eq-partial-z}
{\mathcal L}w_k=\partial_i(A_{ij}\partial_jw_k) = \partial_i((\partial_k A_{ij})w_j).
\end{equation}
Observe that the condition $|\nabla A(x)| \leq K (\delta(x))^{-1}$ implies that the right hand side of \eqref{eq-partial-z} is the divergence of a vector in $L^2$ and 
thus the solutions $w_k$ will belong $W^{1,2}_{\text{loc}}$.
We record the regularity results in the following Proposition.

\begin{proposition}\label{Regularity}
Suppose that  $u\in W^{1,2}_{loc}(\Omega;{\BBC})$ is the weak solution of ${\mathcal L}u=\mbox{\rm div} A(x)\nabla u= 0$ in $\Omega$.
Let $p_0 = \inf \{p>1: \text{$A$ is $p$-elliptic}\}$,
and suppose that $A$ has bounded measurable coefficients satisfying 
\begin{equation}\label{Bcond}
|\nabla A(x)| \leq K (\delta(x))^{-1}, \quad\forall x \in \Omega
\end{equation}
where the constant $K$ is uniform, and $\delta(x)$ denotes the distance of $x$ to the boundary of $\Omega$. 
Then we have the following improvement in the regularity of $\nabla u$. For any $B_{4r}(x)\subset\Omega$ and $\varepsilon>0$ there exists $C_\varepsilon>0$ such that
\begin{equation}\label{RHthm1}
\left(\dint_{B_{r}(x)} |\nabla u|^{p} \,dy\right)^{1/{p}}
\le
C_\varepsilon\left(\dint_{B_{2 r}(x)} |\nabla u|^{q} \,dy\right)^{1/{q}}+\varepsilon \left(\dint_{B_{2 r}(x)} |\nabla u|^{2} \,dy\right)^{1/{2}}
\end{equation}
for all $p,q \in (p_0, \frac{p'_0n}{n-2})$. (Here $p_0'=p_0/(p_0-1)$ and when $n=2$ one can take $p,q\in (p_0,\infty)$.) The constant in the estimate depends on the dimension, the $p$-ellipticity constants, $\Lambda$, $K$ and $\varepsilon>0$ but not on $x\in\Omega$, $r>0$ or $u$. 

It follows that for any boundary ball $\Delta=\Delta_d\subset\partial\Omega$, for any $p,q\in  (p_0, \frac{p'_0n}{n-2})$ and for any allowed aperture parameters $a,a'>0$ there exists $m=m(a,a')>1$ such that
\begin{equation}\label{NN}
\|\tilde{N}^d_{p,a}(\nabla u)\|_{L^r\Delta_d)} \lesssim \|\tilde{N}^{2d}_{q,a'}(\nabla u)\|_{L^r(m\Delta_d)}
\end{equation}
for all $r >0$. We also have for the same range of $p$'s the estimate
\begin{equation}\label{RHthm2}
\left(r^2\dint_{B_{r}(x)} |\nabla \partial_ku|^2|\partial_k u|^{p-2} \,dy\right)^{1/{p}}
\le
C_p \left(\dint_{B_{2 r}(x)} |\nabla u|^{2} \,dy\right)^{1/{2}},
\end{equation}
for all $k=0,1,2,\dots,n-1$.
\end{proposition}

\subsection{Carleson measures}
\label{SS:Car} 

We begin by recalling the definition of a Carleson measure in a domain $\Omega$ as in \eqref{Omega-111}. 
For $P\in{\BBR}^n$, define the ball centered at $P$ with the radius $r>0$ as
\begin{equation}\label{Ball-1}
B_{r}(P):=\{x\in{\BBR}^n:\,|x-P|<r\}.
\end{equation}
Next, given $Q \in \partial\Omega$, by $\Delta=\Delta_{r}(Q)$ we denote the surface ball  
$\partial\Omega\cap B_{r}(Q)$. The Carleson region $T(\Delta_r)$ is then defined by
\begin{equation}\label{tent-1}
T(\Delta_{r}):=\Omega\cap B_{r}(Q).
\end{equation}

\begin{definition}\label{Carleson}
A Borel measure $\mu$ in $\Omega$ is said to be Carleson if there exists a constant $C\in(0,\infty)$ 
such that for all $Q\in\partial\Omega$ and $r>0$
\begin{equation}\label{CMC-1}
\mu\left(T(\Delta_{r})\right)\leq C\sigma(\Delta_{r}),
\end{equation}
where $\sigma$ is the surface measure on $\partial\Omega$. 
The best possible constant $C$ in the above estimate is called the Carleson norm 
and is denoted by $\|\mu\|_{\mathcal C}$.
\end{definition}

In all that follows we now assume that the coefficients of the matrix $A$ and $B$ of the elliptic operator $\mathcal{L}=\mbox{div} A(x)\nabla +B(x)\cdot\nabla$  
satisfies the following natural conditions.  
First, we assume that the entries $A_{ij}$ of $A$ are in ${\rm Lip}_{loc}(\Omega)$ and the entries of $B$ are $L^\infty_{loc}(\Omega)$. 
Second, we assume that
\begin{equation}\label{CarA}
d\mu(x)=\sup_{B_{\delta(x)/2}(x)}[|\nabla A|^2+|B|^2]\delta(x) \,dx
\end{equation}
is a Carleson measure in $\Omega$. Sometimes, and for certain coefficients of $A$, we will
assume that their Carleson norm $\|\mu\|_{\mathcal{C}}$ is sufficiently small. 
The fact that $\mu$ is a Carleson allows one to relate integrals in $\Omega$ with respect to $\mu$ to boundary integrals involving 
the nontangential maximal function.
We have the following result for our averaged nontangential maximal function (c.f. \cite{DPcplx}).

\begin{theorem}\label{T:Car}
Suppose that $d\nu=f\,dx$ and $d\mu(x)=\left[\sup_{B_{\delta(x)/2}(x) }|f|\right]dx$. Assume that
$\mu$ is a Carleson measure. Then there exists a finite 
constant $C=C(L,a)>0$ such that for every $u\in L^{p}_{loc}(\Omega;{\BBC})$ one has
\begin{equation}\label{Ca-222}
\int_{\Omega}|u(x)|^p\,d\nu(x)\leq C\|\mu\|_{\mathcal{C}} 
\int_{\partial\Omega}\left(\tilde{N}_{p,a}(u)\right)^p\,d\sigma.
\end{equation}
Furthermore, consider $\Omega={\mathbb R}^n_+$ where  $\mu$ and $\nu$ are measures as above supported in $\Omega$ and $\delta(x_0,x')=x_0$. Let
 $h:{\mathbb R}^{n-1}\to {\mathbb R}^+$ be a Lipschitz function with Lipschitz norm $L$
and 
$$\Omega_h=\{(x_0,x'):x_0>h(x')\}.$$
Then for any $\Delta\subset {\mathbb R}^{n-1}$ with $\sup_{\Delta} h\le \mbox{diam}(\Delta)/2$ we have
\begin{equation}\label{Ca-222-x}
\int_{\Omega_h\cap T(\Delta)}|u(x)|^p\,d\nu(x)\leq C\|\mu\|_{\mathcal{C}} 
\int_{\partial\Omega_h\cap T(\Delta)}\left(\tilde{N}_{p,a,h}(u)\right)^p\,d\sigma.
\end{equation}
Here for a point $Q=(h(x'),x')\in\partial\Omega_h$ we define
\begin{equation}
\tilde{N}_{p,a,h}(u)(Q) = \sup_{\Gamma_a(Q)}w,\label{eq-Nh}
\end{equation}
where
\begin{equation}\label{TFC-6x}
\Gamma_{a}(Q)=\Gamma_{a}((h(x'),x'))=\{y=(y_0,y')\in\Omega:\,a|h(x')-y_0|>|x'-y'|\}
\end{equation}
and the $L^p$ averages $w$ are defined by \eqref{w} where the distance $\delta$ is taken with respect to the domain $\Omega={\mathbb R}^n_+$.
\end{theorem}

\subsection{The $L^p$-Dirichlet problem}

We recall the definition of $L^p$ solvability of the Dirichlet problem. 
When an operator $\mathcal L$ is as in Theorem \ref{S3:T1} is uniformly elliptic (i.e. $2$-elliptic)
the Lax-Milgram lemma can be applied and guarantees the existence of weak solutions.
That is, 
given any $f\in \dot{B}^{2,2}_{1/2}(\partial\Omega;{\BBC})$, the homogenous space of traces of functions in $\dot{W}^{1,2}(\Omega;{\BBC})$, there exists a unique (up to a constant)
$u\in \dot{W}^{1,2}(\Omega;{\BBC})$ such that $\mathcal{L}u=0$ in $\Omega$ and ${\rm Tr}\,u=f$ on $\partial\Omega$. We call these solutions \lq\lq  energy solutions" and use them to define the notion of solvability of the $L^p$ Dirichlet problem.

\begin{definition}\label{D:Dirichlet} 
Let $\Omega$ be the Lipschitz domain introduced in \eqref{Omega-111} and fix an integrability exponent 
$p\in(1,\infty)$. Also, fix an aperture parameter $a>0$. Consider the following Dirichlet problem 
for a complex valued function $u:\Omega\to{\BBC}$:
\begin{equation}\label{E:D}
\begin{cases}
0=\partial_{i}\left(A_{ij}(x)\partial_{j}u\right) 
+B_{i}(x)\partial_{i}u 
& \text{in } \Omega,
\\[4pt]
u(x)=f(x) & \text{ for $\sigma$-a.e. }\,x\in\partial\Omega, 
\\[4pt]
\tilde{N}_{2,a}(u) \in L^{p}(\partial \Omega), &
\end{cases}
\end{equation}
where the usual Einstein summation convention over repeated indices ($i,j$ in this case) 
is employed. 

We say the Dirichlet problem \eqref{E:D} is solvable for a given $p\in(1,\infty)$ if  there exists a
$C=C(p,\Omega)>0$ such that
for all boundary data
$f\in L^p(\partial\Omega;{\BBC})\cap \dot{B}^{2,2}_{1/2}(\partial\Omega;{\BBC})$ the unique energy solution
satisfies the estimate
\begin{equation}\label{y7tGV}
\|\tilde{N}_{2,a} (u)\|_{L^{p}(\partial\Omega)}\leq C\|f\|_{L^{p}(\partial\Omega;{\BBC})}.
\end{equation}

Similarly, we say the Regularity problem for the same PDE is solvable for a given $p\in(1,\infty)$ if  there exists a
$C=C(p,\Omega)>0$ such that
for all boundary data
$f\in \dot{W}^{1,p}(\partial\Omega;{\BBC})\cap \dot{B}^{2,2}_{1/2}(\partial\Omega;{\BBC})$ the unique (modulo constants) energy solution
satisfies the estimate
\begin{equation}\label{y7tGVx}
\|\tilde{N}_{2,a} (\nabla u)\|_{L^{p}(\partial\Omega)}\leq C\|\nabla_Tf\|_{L^{p}(\partial\Omega;{\BBC})}.
\end{equation}

\end{definition}

\noindent{\it Remark.}  Given $f\in\dot{B}^{2,2}_{1/2}(\partial\Omega;{\BBC})\cap L^p(\partial\Omega;{\BBC})$
the corresponding energy solution constructed above is unique (since the decay implied by the $L^p$ estimates eliminates constant solutions). As the space 
$\dot{B}^{2,2}_{1/2}(\partial\Omega;{\BBC})\cap L^p(\partial\Omega;{\BBC})$ is dense in 
$L^p(\partial\Omega;{\BBC})$ for each $p\in(1,\infty)$, it follows that there exists a 
unique continuous extension of the solution operator
$f\mapsto u$
to the whole space $L^p(\partial\Omega;{\BBC})$, with $u$ such that $\tilde{N}_{2,a} (u)\in L^p(\partial\Omega)$ 
and the accompanying estimate $\|\tilde{N}_{2,a} (u) \|_{L^{p}(\partial \Omega)} 
\leq C\|f\|_{L^{p}(\partial\Omega;{\BBC})}$ being valid. Furthermore, as shown in the Appendix of \cite{DPcplx} for any $f\in L^p(\partial \Omega;\mathbb C)$ the corresponding solution $u$ constructed by the continuous extension attains the datum $f$ as its boundary values in the following sense.
Consider the average $\tilde u:\Omega\to \mathbb C$ defined by
$$\tilde{u}(x)=\dint_{B_{\delta(x)/2}(x)} u(y)\,dy,\quad \forall x\in \Omega.$$
Then 
\begin{equation}
f(Q)=\lim_{x\to Q,\,x\in\Gamma(Q)}\tilde u(x),\qquad\text{for a.e. }Q\in\partial\Omega,
\end{equation}
where the a.e. convergence is taken with respect to the ${\mathcal H}^{n-1}$ Hausdorff measure on $\partial\Omega$.

We can make a similar statement regarding nontangential convergence of gradients for solutions to the Regularity problem.
That is, defining
$${\tilde \nabla u}(x)=\dint_{B_{\delta(x)/2}(x)} \nabla u(y)\,dy,\quad \forall x\in \Omega,$$
the same proof in \cite{DPcplx} yields that

\begin{equation}
\nabla u(Q) =\lim_{x\to Q,\,x\in\Gamma(Q)}\tilde \nabla u(x),\qquad\text{for a.e. }Q\in\partial\Omega,
\end{equation}

and when $\Omega=\BBR^n_+$,

\begin{equation}
\nabla_T f(Q)=\lim_{x\to Q,\,x\in\Gamma(Q)}\tilde \nabla_Tu(x),\qquad\text{for a.e. }Q\in\partial\Omega,
\end{equation}

\vskip2mm
Let us make some observations that explain the structural assumptions we have made in Theorems \ref{S3:T0} and \ref{S3:T1}.
As we have already stated it suffices to formulate the result in the case $\Omega={\mathbb R}^n_+$ by using the pull-back map introduced above. Since Theorem \ref{S3:T1} requires that the coefficients have {\it small} Carleson norm this puts a restriction on the size of the Lipschitz constant $L=\|\nabla\phi\|_{L^\infty}$ of the map $\phi$ that defines the domain $\Omega$ in \eqref{Omega-111}.
The constant $L$ will have also to be small (depending on $\lambda_p$, $\Lambda$, $n$ and $p$). \vglue1mm

For technical reasons in the proof we also need that all coefficients $A_{0j}$, $j=0,1,\dots,n-1$ are real. This can be ensured as follows. When $j>0$ observe that we have 
\begin{equation}
\partial_0([\mathscr{I}m\,A_{0j}]\partial_j u)=\partial_j([\mathscr{I}m\,A_{0j}]\partial_0 u)+(\partial_0 [\mathscr{I}m\,A_{0j}])\partial_ju-([\partial_j \mathscr{I}m\,A_{0j}])\partial_0u\label{eqSWAP}
\end{equation}
which allows to move the imaginary part of the coefficient $A_{0j}$ onto the coefficient $A_{j0}$ at the expense of two (harmless) first order terms. This does not work for the coefficient $A_{00}$. Instead we make the following observation. 

Suppose that the measure \eqref{CarA}  associated to an operator $ \mathcal L = \partial_{i}\left(A_{ij}(x)\partial_{j}\right)  +B_{i}(x)\partial_{i}$  is Carleson. 
Consider a related operator 
$ \tilde{\mathcal L} = \partial_{i}\left(\tilde{A}_{ij}(x)\partial_{j}\right) 
+\tilde{B}_{i}(x)\partial_{i}$,
where $\tilde A = \alpha{A}$ and $\tilde B=\alpha{B} - (\partial_{i}\alpha){A}_{ij}\partial_j$, and
$\alpha\in L^\infty(\Omega)$ is a complex valued function such that $|\alpha(x)|\ge \alpha_0>0$ and $|\nabla\alpha|^2x_0$ is a Carleson measure. 
 
Observe that a weak solution $u$ to $\tilde{\mathcal L}u=0$ is also a weak solution to $\mathcal Lu=0$ and that
the new coefficients of $\tilde A$ and $\tilde B$ also satisfy a Carleson measure condition as in \eqref{CarA}, from the assumption on $\alpha$. We will only require that the coefficient ${\tilde A}_{00}$ is real but we may as well ensure for simplicity that it equals to $1$. Clearly, if we choose $\alpha = {A}_{00}^{-1}$, then the new operator $\tilde L$ will have this property. When ${A}_{00}$ (and hence $\alpha$) is real, then  $\tilde{A}$. Similarly, if ${A} $ is $p$-elliptic and $\mathscr{I}m\,{A}_{00}$ is 
sufficiently small (depending on the ellipticity constants), then $\tilde A$ will also be  $p$-elliptic. However, if $\mathscr{I}m\,\alpha$ is not small,
the $p$-ellipticity, after multiplication of ${A}$ by $\alpha$ may not be preserved. Thus, we assume in our main results (Theorems \ref{S3:T0} and \ref{S3:T1}) the 
$p$-ellipticity of the new matrix  $\tilde A$ which has all coefficients $\tilde A_{0j}$, $j=0,1,\dots,n-1$ real, as this is not implied in the general case from the $p$-ellipticity of the original matrix $A$.

\section{Bounds for the nontangential maximal function by the square function}
\label{SS:43}

We work on $\Omega=\BBR^n_+$ and we assume that the matrix $A$ is $p$-elliptic.  
Our aim in this section is to establish  bounds for the nontangential maximal function by the square function.
The approach necessarily differs from the usual argument in the real scalar elliptic case due to the fact 
that certain estimates, such as interior H\"older regularity of a weak solution, are unavailable for 
the complex coefficient case. Here we deviate from the approach take in \cite{DPcplx} where we worked with $p$-adapted square function and instead focus on the estimates for the usual square function. Our approach is similar to \cite{DHM} for elliptic systems and when possible we refer to result from there.

The major innovation from \cite{DHM} is the use of an entire family of Lipschitz graphs on which the nontangential 
maximal function is large in lieu of a single graph constructed via a stopping time argument. 
This is necessary as we are using $L^2$ averages of solutions to define the nontangential maximal 
function and hence the knowledge of certain bounds for a solution on a single graph provides no 
information about the $L^2$ averages over interior balls.

Let $u$ be an energy solution to
$${\mathcal L}u=\partial_i(A_{ij}\partial_ju)=0,\qquad 
\mbox{in }\Omega={\mathbb R}^n_+.$$
Let $v=\nabla u$, that is $v_k=\partial_ku$, $k=0,1,\dots,n-1$. Let $w=w_2$ be the $L^2$ averages of $v$, that is
\begin{equation}\label{w_new}
w(x):=\left(\dint_{B_{\delta(x)/2}(x)}|v(z)|^{2}\,dz\right)^{1/2}.
\end{equation}

Set
\begin{equation}\label{E}
E_{\nu,a}:=\big\{x'\in\partial\Omega:\,N_{a}(w)(x')>\nu\big\}
\end{equation}
(where, as usual, $a>0$ is a fixed background parameter), 
and consider the map $h:\partial\Omega\to\BBR$ given at each $x'\in\partial\Omega$ by 
\begin{equation}\label{h}
h_{\nu,a}(w)(x'):=\inf\left\{x_0>0:\,\sup_{z\in\Gamma_{a}(x_0,x')}w(z)<\nu\right\}
\end{equation}
with the convention that $\inf\varnothing=\infty$. 
We remark that $h$ differs somewhat from the function that has been used in the argument for scalar equations 
(cf. \cite[pp.\,212]{KP01} and \cite{KKPT}). 

At this point we note that $h_{\nu,a}(w,x')<\infty$ for all points $x'\in\partial\Omega$. Since $u\in \dot{W}^{1,2}(\mathbb R^n_+;\mathbb C)$ it follows that $v\in L^2(\mathbb R^n_+;\mathbb C^n)$. Thus $w$ as an $L^2$ average of $v$ is continuous on the upper half-space and decays to zero as $x_0\to\infty$. Thus $h_{\nu,a}$ is finite everywhere.

We look at some further properties of this function. As in \cite{DHM} we have the following (with identical proof).
\begin{lemma}\label{S3:L5}
Let $w$ be as above \eqref{w_new}. 
Also, fix two positive numbers $\nu,a$. Then the following properties hold.
\vglue2mm

\noindent (i)
The function $h_{\nu, a}(w)$ is Lipschitz, with a Lipschitz constant $1/a$. That is,
\begin{equation}\label{Eqqq-5}
\left|h_{\nu,a}(w)(x')-h_{\nu,a}(w)(y')\right|\leq a^{-1}|x'-y'|
\end{equation}
for all $x',y'\in\partial\Omega$.

\vglue2mm

\noindent (ii)
Given an arbitrary $x'\in E_{\nu,a}$, let $x_0:=h_{\nu,a}(w)(x')$. Then there exists a 
point $y=(y_0,y')\in\partial\Gamma_{a}(x_0,x')$ such that $w(y)=\nu$ and $h_{\nu,a}(w)(y')=y_0$. 		
\end{lemma}

We also have (as in \cite{DHM}) by an identical argument:

\begin{lemma}\label{l6} 
Let $v,\,w$ be as above.
For any $a>0$ there exists $b=b(a)>a$ and $\gamma=\gamma(a)>0$ such that the following holds. 
Having fixed an arbitrary $\nu>0$, for each point $x'$ from the set 
\begin{equation}\label{Eqqq-17}
\big\{x':\,N_{a}(w)(x')>\nu\mbox{ and }S_{b}(v)(x')\leq\gamma\nu\big\}
\end{equation}
there exists a boundary ball $R$ with $x'\in 2R$ and such that
\begin{equation}\label{Eqqq-18}
\big|w\big(h_{\nu,a}(w)(z'),z'\big)\big|>\nu/{2}\,\,\text{ for all }\,\,z'\in R.
\end{equation}
Here $S_b=S_{2,b}$ is the usual square function of $v=\nabla u$ associated with nontangential cones $\Gamma_b(.)$.
\end{lemma}

Given a Lipschitz function $h:{\mathbb{R}}^{n-1}\to{\mathbb{R}}$, denote by $M_h$ the 
Hardy-Littlewood maximal function considered on the graph of $h$. That is, 
given any locally integrable function $f$ on the Lipschitz surface 
$\Lambda_h=\{(h(z'),z'):\,z'\in\BBR^{n-1}\}$, define 
$(M_h f)(x):=\sup_{r>0}\dint_{\Lambda_h\cap B_r(x)}|f|\,d\sigma$ for each $x\in\Lambda_h$. 

\begin{corollary}\label{S3:L6} 
Let $v,w$ be defined as above and let $a>0$ be fixed. 
Associated with these, let $b,\,\gamma$ be as in Lemma~\ref{l6}. Then there exists a finite 
constant $C=C(n,p)>0$ with the property that for any $\nu>0$ and any point $x'\in E_{\nu,a}$ 
such that $S_{b}(v)(x')\leq\gamma\nu$ one has
\begin{equation}\label{Eqqq-23}
(M_{h_{\nu,a}}w)\big(h_{\nu,a}(x'),x'\big)\geq\,C\nu.
\end{equation}
\end{corollary}

The following lemma requires a modified proof which we include below.

\begin{lemma}\label{S3:L8} 
Consider the equation ${\mathcal Lu}=0$ with coefficients satisfying assumptions of Theorem~\ref{S3:T1}, let $v=\nabla u$ and let $w$ be defined by \eqref{w_new}. 
Then there exists $a>0$ with the following significance.  Select $\theta\in[1/6,6]$ and, having picked $\nu>0$ arbitrary,  
let $h_{\nu,a}(w)$ be as in \eqref{h}. Also, consider the domain 
$\mathcal{O}=\{(x_0,x')\in\Omega:\,x_0>\theta h_{\nu,a}(x')\}$ with boundary  
$\partial\mathcal{O}=\{(x_0,x')\in\Omega:\,x_0=\theta h_{\nu,a}(x')\}$. In this context, 
for any surface ball $\Delta_r=B_r(Q)\cap\partial\Omega$, with $Q\in\partial\Omega$ and $r>0$ 
chosen such that $h_{\nu,a}(w)\leq 2r$ pointwise on $\Delta_{2r}$, 
one has
\begin{align}\label{TTBBMM}
\int_{\Delta_r}\big|v\big(\theta h_{\nu,a}(w)(\cdot),\cdot\big)\big|^2\,dx' 
&\leq C(1+\|\mu\|^{1/2}_{\mathcal C})\|S_{b}(v)\|_{L^p(\Delta_{2r})}
\|{N}_{2,a}(w)\|_{L^p(\Delta_{2r})}
\nonumber\\
&\hskip-3cm +C\|\mu\|_{\mathcal C}^{1/2}\|{N}_{2,a}(w)\|^2_{L^p(\Delta_{2r})}+C\|S_{b}(v)\|^2_{L^p(\Delta_{2r})}+\frac{c}{r}\iint_{\mathcal{K}}|v|^{2}\,dX.
\end{align}
Here $C=C(\Lambda,p,n)\in(0,\infty)$ and $\mathcal{K}$ is a region inside $\mathcal{O}$ of diameter, 
distance to the boundary $\partial\mathcal{O}$, and distance to $Q$, are all comparable to $r$. 
Also, the parameter $b>a$ is as in Lemma~\ref{l6}, and the cones used to define the square and nontangential 
maximal functions in this lemma have vertices on $\partial\Omega$.

Moreover, the term $\displaystyle\iint_{\mathcal{K}}|v|^2\,dX$ appearing 
in \eqref{TTBBMM} may be replaced by the quantity
\begin{equation}\label{Eqqq-25}
Cr^{n-1}|\tilde{v}(A_r)|^2+C\int_{\Delta_{2r}}S^2_{b}(v)\,d\sigma,
\end{equation}
where $A_r$ is any point inside $\mathcal{K}$ (usually called a corkscrew point of $\Delta_r$) and
\begin{equation}\label{Eqqq-26}
\tilde{v}(X):=\dint_{B_{\delta(X)/2}(X)}v(Z)\,dZ.
\end{equation}
Finally, \eqref{TTBBMM} and \eqref{Eqqq-25} remains true even if $v$ is replaced by $v-v_0$ for any fixed $v_0\in{\mathbb C}^n$.
\end{lemma}

\begin{proof} 
Fix $\theta\in [1/6,6]$. Consider the well-known pullback transformation $\rho:\BBR^{n}_{+}\to\mathcal{O}$ 
appearing in works of Dahlberg, Ne\v{c}as, 
Kenig-Stein and others, defined by
\begin{equation}\label{E:rho}
\rho(x_0, x'):=\big(x_0+P_{\gamma x_0}\ast\phi(x'),x'\big),
\qquad\forall\,(x_0,x')\in\mathbb{R}^{n}_{+},
\end{equation}
for some positive constant $\gamma$. Here $\phi$ is a Lipschitz function describing boundary on $\partial\mathcal O$, $P$ is a nonnegative function 
$P\in C_{0}^{\infty}(\mathbb{R}^{n-1})$ and, for each $\lambda>0$,  
\begin{equation}\label{PPP-1a}
P_{\lambda}(x'):=\lambda^{-n+1}P(x'/\lambda),\qquad\forall\,x'\in{\mathbb{R}}^{n-1}.
\end{equation}
Finally, $P_{\lambda}\ast\phi(x')$ is the convolution
\begin{equation}\label{PPP-lambda}
P_{\lambda}\ast\phi(x'):=\int_{\mathbb{R}^{n-1}}P_{\lambda}(x'-y')\phi(y')\,dy'. 
\end{equation}
Observe that $\rho$ extends up to the boundary of ${\BBR}^{n}_{+}$ and maps one-to-one from 
$\partial {\BBR}^{n}_{+}$ onto $\partial\mathcal O$. Also for sufficiently small $\gamma\lesssim L$ 
the map $\rho$ is a bijection from $\overline{\mathbb{R}^{n}_{+} }$ onto $\overline{\mathcal O}$ 
and, hence, invertible. 

For a solution $u\in W^{1,2}_{loc}(\mathcal O;\BBC)$ to $\mathcal{L}u=0$ in $\mathcal O$ with Dirichlet 
datum $f$, consider $\tilde{u}:=u\circ\rho$ and $\tilde{f}:=f\circ\rho$. The change of variables 
via the map $\rho$ just described implies that $\tilde{u}\in W^{1,2}_{loc}(\mathbb{R}^{n}_{+};{\BBC})$ 
solves a new elliptic PDE of the form
\begin{equation}\label{ESv}
\partial_{i}\left(\tilde{A}_{ij}(x)\partial_{j}{\tilde{u}}\right)
=0,
\end{equation}
with boundary datum $\tilde{f}$ on $\partial \mathbb{R}^{n}_{+}$. Hence, solving a boundary value 
problem for $u$ in $\Omega$ is equivalent to solving a related boundary value problem for $\tilde{u}$ in 
$\mathbb{R}^{n}_{+}$. Crucially, if the coefficients of the original system are such that \eqref{CarA} 
is a Carleson measure, then the coefficients of $\tilde{A}$ satisfy an analogous 
Carleson condition in the upper-half space. If, in addition, the Carleson norm of \eqref{CarA} 
is small and $L$ (the Lipschitz constant for the domain $\Omega$) is also small, then the Carleson 
norm for the new coefficients $\tilde{A}$ 
\begin{equation}\label{CarbarA}
d\tilde{\mu}(x)=\left(\sup_{B_{\delta(x)/2}(x)}|\nabla\tilde{A}|\right)^{2}\delta(x) \,dx
\end{equation}
will be correspondingly small ans will only depends on the Carleson norm of the original coefficients and the Lipschitz norm of the function $h_{\nu,a}$. When the Lipschitz norm of this function goes to zero we have 
$$\limsup \|\tilde\mu\|_{\mathcal{C}}\le \|\mu\|_{\mathcal{C}}$$
and hence the parameter $a>0$ may be chosen 
large enough so that the Lipschitz norm of the function $\theta h_{\nu,a}$ is sufficiently small (at most $6/a$)
such that $\|\tilde\mu\|_{\mathcal{C}}\le 2\|\mu\|_{\mathcal{C}}$.  
Moreover, this transformation also preserves ellipticity.

\vskip1mm

Having fixed a scale $r>0$, we localize to a ball $B_r(y')$ in $\BBR^{n-1}$. 
Let $\zeta$ be a smooth cutoff function of the form $\zeta(x_0, x')=\zeta_{0}(x_0)\zeta_{1}(x')$ where
\begin{equation}\label{Eqqq-27}
\zeta_{0}= 
\begin{cases}
1 & \text{ in } [0, r], 
\\
0 & \text{ in } [2r, \infty),
\end{cases}
\qquad
\zeta_{1}= 
\begin{cases}
1 & \text{ in } B_{r}(y'), 
\\
0 & \text{ in } \mathbb{R}^{n}\setminus B_{2r}(y')
\end{cases}
\end{equation}
and
\begin{equation}\label{Eqqq-28}
r|\partial_{0}\zeta_{0}|+r|\nabla_{x'}\zeta_{1}|\leq c
\end{equation}
for some constant $c\in(0,\infty)$ independent of $r$.
Our goal is to control the $L^2$ norm of $\nabla u\big(\theta h_{\nu,a}(w)(\cdot),\cdot\big)$.  
Since after the pullback under the mapping $\rho$ the latter is comparable with the $L^2$ norm 
of $\nabla \tilde{u}(0,\cdot)$, we proceed to estimate this quantity.

Clearly, if we establish estimate \eqref{TTBBMM} for $\nabla \tilde{u}$ on $\Delta_r\subset {\partial{\mathbb R}^n_+}$ it would imply the original estimate for $\nabla u$ on the graph of $\theta h_{\nu,a}$.

Hence, let $\tilde{v}=\nabla\tilde{u}$. For $\tilde{v}_k$ where $k=1,2,\dots,n-1$ we have
\begin{align}\label{u6tg}
&\hskip -0.20in
\int_{B_{2r}(y')}|\tilde{v}_k|^{2}(0,x')\zeta(0,x')\,dx' 
\nonumber\\[4pt]
&\hskip 0.70in
=-\iint_{[0,2r]\times B_{2r}(y')}\partial_{0}\left[|\tilde{v}_k|^{2}\zeta\right](x_0,x')\,dx_0\,dx' 
\nonumber\\[4pt]
&\hskip 0.70in
=-2\iint_{[0,2r]\times B_{2r}(y')}\mathscr{R}e\,\langle \tilde{v}_k,\partial_0 \tilde{v}_k\rangle\zeta\,dx_0\,dx'  
\nonumber\\[4pt]
&\hskip 0.70in
\quad-\iint_{[0,2r]\times B_{2r}(y')}|\tilde{v}_k|^{2}(x_0,x')\partial_{0}\zeta\,dx_0\,dx'
\nonumber\\[4pt]
&\hskip 0.70in
=:\mathcal{A}+IV.
\end{align}
We further expand the term $\mathcal A$ as a sum of three terms obtained 
via integration by parts with respect to $x_0$ as follows:
\begin{align}\label{utAA}
\mathcal A &=-2\iint_{[0,2r]\times B_{2r}(y')}\mathscr{R}e\,\langle \tilde{v}_k,\partial_0 \tilde{v}_k\rangle\zeta(\partial_0x_0)\,dx_0\,dx'
\nonumber\\[4pt]
&=\quad 2\iint_{[0,2r]\times B_{2r}(y')}\left|\partial_{0}\tilde{v}_k\right|^{2}x_0\zeta\,dx_0\,dx' 
\nonumber\\[4pt]
&\quad +2\iint_{[0,2r]\times B_{2r}(y')}\mathscr{R}e\,\langle \tilde{v}_k,\partial^2_{00} \tilde{v}_k\rangle x_0\zeta\,dx_0\,dx' 
\nonumber\\[4pt]
&\quad +2\iint_{[0,2r]\times B_{2r}(y')}\mathscr{R}e\,\langle \tilde{v}_k,\partial_0 \tilde{v}_k\rangle x_0\partial_{0}\zeta\,dx_0\,dx' 
\nonumber\\[4pt]
&=:I+II+III.
\end{align}

We start by analyzing the term $II$. We write $\partial^2_0\tilde{v}_k=\partial_k\partial_0\tilde{v}_0$ and integrate by parts moving the $\partial_k$ derivative. This gives
\begin{align}\label{utAAA}
II &=2\iint_{[0,2r]\times B_{2r}(y')}\mathscr{R}e\,\langle \tilde{v}_k,\partial_k\partial_0\tilde{v}_0\rangle x_0\zeta\,dx_0\,dx' 
\nonumber\\[4pt]
&=-2\iint_{[0,2r]\times B_{2r}(y')}\mathscr{R}e\,\langle \partial_k\tilde{v}_k,\partial_0\tilde{v}_0\rangle x_0\zeta\,dx_0\,dx' 
\nonumber\\[4pt]
&\quad -2\iint_{[0,2r]\times B_{2r}(y')}\mathscr{R}e\,\langle \tilde{v}_0,\partial_k \tilde{v}_k\rangle x_0\partial_{0}\zeta\,dx_0\,dx' 
\nonumber\\[4pt]
&=II_1+II_2.
\end{align}

We now group together terms that are of the same type. Firstly, we have 
\begin{equation}\label{Eqqq-29}
I+II_1\leq C(\Lambda,n)\|S_{b}(v)\|^2_{L^2(B_{2r})}.
\end{equation}
Here,  the estimate would be true even with truncated square function $\|S^{2r}_{b}(\tilde{v})\|^2_{L^2(B_{2r})} $ which is at every point
dominated by $\|S_{b}(v)\|^2_{L^2(B_{2r})}$.

Next, corresponding to the case when the derivative falls on the cutoff function $\zeta$ we have
\begin{align}\label{TDWW}
II_2+III &\leq C(\Lambda,n)\iint_{[0,2r]\times B_{2r}}\left|\nabla \tilde{v}\right||\tilde{v}|\frac{x_0}{r}\,dx_0\,dx' 
\nonumber\\[4pt]
&\leq C(\Lambda,n)\left(\iint_{[0,2r]\times B_{2r}}|\tilde{v}|^{2}\frac{x_0}{r^{2}}\,dx_0\,dx'\right)^{1/2} 
\|S^{2r}_{b}(\tilde{v})\|_{L^2(B_{2r})} 
\nonumber\\[4pt]
&\leq C(\Lambda,n)\|S_{b}(v)\|^{p/2}_{L^p(B_{2r})}\|{N}_{p,a}(w)\|_{L^2(B_{2r})}.
\end{align}
Finally, the interior term $V$, which arises from the fact that $\partial_{0}\zeta$ vanishes on the set
$(0,r)\cup(2r,\infty)$ may be estimated as follows:
\begin{equation}\label{Eqqq-31}
IV\leq\frac{c}{r}\iint_{[r,2r]\times B_{2r}}|v|^{2}\,dx_0\,dx'.
\end{equation}
Summing up all terms, the above analysis ultimately yields
\begin{align}\label{E1:uonh}
&\hskip -0.20in \int_{B_{r}(y')}|\nabla_T\tilde{u}(0,x')|^2\,dx' 
\nonumber\\[4pt]
&\hskip 0.40in 
\leq C(\Lambda,n)(1+\|\mu\|^{1/2}_{\mathcal C}) 
\|S_{b}(v)\|_{L^p(B_{2r})}\|{N}_a(w)\|_{L^p(B_{2r})}
\nonumber\\[4pt]
&\hskip 0.40in 
\quad+C(\Lambda,n)\|S_{b}(v)\|^2_{L^p(B_{2r})}
+\frac{c}{r}\iint_{[r,2r]\times B_{2r}}|v|^2\,dx_0\,dx'.
\end{align}
Observe also we could have done the whole calculation with a constant subtracted off $\tilde{v}_k$ without any substantial modifications. It remains to consider derivative in a transversal direction to the boundary.
Instead of $\tilde{v}_0=\partial_0\tilde{u}$ it is more convenient to work with
$$H=\sum_{j=0}^{n-1}\tilde{A}_{0j}\tilde{v}_j,$$
which will give us desired bound since
\begin{align}\label{u6tg-otoh}
&
\int_{B_{2r}(y')}|\tilde{v}_0|^{2}(0,x')\zeta(0,x')\,dx' \approx \int_{B_{2r}(y')}|\tilde{A}_{00}\tilde{v}_0(0,x')|^{2}\zeta(0,x')\,dx' 
\nonumber\\
&
\le n\left[\int_{B_{2r}(y')}|H(0,x')|^{2}\zeta(0,x')\,dx'+\sum_{j>0}\int_{B_{2r}(y')}|\tilde{A}_{0j}\tilde{v}_j(0,x')|^{2}\zeta(0,x')\,dx'  \right]
\nonumber\\
&
\le n\int_{B_{2r}(y')}|H(0,x')|^{2}\zeta(0,x')\,dx'+C(n,\Lambda) \int_{B_{r}(y')}|\nabla_T\tilde{u}(0,x')|^2\,dx' .
\end{align}
The second term is OK as we have \eqref{E1:uonh}. We deal with the first term now. A calculation similar to 
\eqref{u6tg}-\eqref{utAA} gives us
\begin{align}\label{u6tg-x}
&\hskip -0.20in
\int_{B_{2r}(y')}|H|^{2}(0,x')\zeta(0,x')\,dx' 
\nonumber\\[4pt]
&\hskip 0.70in
=-2\iint_{[0,2r]\times B_{2r}(y')}\mathscr{R}e\,\langle H,\partial_0 H\rangle\zeta\,dx_0\,dx' 
\nonumber\\[4pt]
&\hskip 0.70in
\quad-\iint_{[0,2r]\times B_{2r}(y')}|H|^{2}(x_0,x')\partial_{0}\zeta\,dx_0\,dx'.
\end{align}
The second term has a similar estimate as \eqref{Eqqq-31}. For the first term we use the fact that ${\tilde{\mathcal L}\tilde{u}}=0$ which implies that
$$\partial_0 H=-\sum_{i>0}\partial_i(\tilde{A}_{ij}\tilde{v}_j).$$
It follows
\begin{align}\label{u6tg-xx}
&\hskip -0.20in
-2\iint_{[0,2r]\times B_{2r}(y')}\mathscr{R}e\,\langle H,\partial_0 H\rangle\zeta\,dx_0\,dx' 
\nonumber\\[4pt]
&\hskip 0.20in
=2\sum_{i>0}\iint_{[0,2r]\times B_{2r}(y')}\mathscr{R}e\,\langle H,\partial_i (\tilde{A}_{ij}\tilde{v}_j)\rangle\zeta(\partial_0x_0)\,dx_0\,dx' 
\nonumber\\[4pt]
&\hskip 0.20in
=-2\sum_{i>0}\iint_{[0,2r]\times B_{2r}(y')}\mathscr{R}e\,\langle\partial_0 H,\partial_i (\tilde{A}_{ij}\tilde{v}_j)\rangle\zeta x_0\,dx_0\,dx'
\nonumber\\[4pt]
&\hskip 0.20in
\quad+2\sum_{i>0}\iint_{[0,2r]\times B_{2r}(y')}\mathscr{R}e\,\langle\partial_i H,\partial_0 (\tilde{A}_{ij}\tilde{v}_j)\rangle\zeta x_0\,dx_0\,dx'
\nonumber\\[4pt]
&\hskip 0.20in
\quad-2\sum_{i>0}\iint_{[0,2r]\times B_{2r}(y')}\mathscr{R}e\,\langle H,\partial_i (\tilde{A}_{ij}\tilde{v}_j)\rangle(\partial_0\zeta) x_0\,dx_0\,dx'
\nonumber\\[4pt]
&\hskip 0.20in
\quad+2\sum_{i>0}\iint_{[0,2r]\times B_{2r}(y')}\mathscr{R}e\,\langle H,\partial_0 (\tilde{A}_{ij}\tilde{v}_j)\rangle(\partial_i\zeta) x_0\,dx_0\,dx'.
\end{align}

We analyse this term by term. In the last two terms, if the derivative falls on $\tilde{v}_j$ these terms are of the same nature as \eqref{TDWW} and are handled identically. When the derivative falls on the coefficients these are bounded by
$$\iint_{[0,2r]\times B_{2r}(y')}|\tilde{v}|^2|\nabla \tilde{A}|\frac{x_0}r\,dx_0\,dx'\lesssim \|\mu\|^{1/2}_{\mathcal C}\|N_a(w)\|^2_{L^2},$$
where we have used the Cauchy-Schwarz inequality and the Carleson condition. 

The first two terms on the righthand side of \eqref{u6tg-xx} will give us the square function of $\tilde{v}$ when both derivatives fall on $\tilde{v}$ or a mixed term like \eqref{TDWW} or finally when both derivatives hit the coefficients terms bounded from above by
$$\iint_{[0,2r]\times B_{2r}(y')}|\tilde{v}|^2|\nabla \tilde{A}|^2x_0\,dx_0\,dx'\lesssim \|\mu\|_{\mathcal C}\|N_a(w)\|^2_{L^2}.$$

With this in hand, the estimate in \eqref{TTBBMM} follows (by passing from $\tilde{v}$ back to $v=\nabla u$ via the map $\rho$).

Finally, the claim that the term \eqref{Eqqq-25} can be used in the statement of the lemma follows from Poincar\'e inequality. See \cite{DHM} for the details.
\end{proof}

Finally, by using all Lemmas above we can establish the following local good-$\lambda$ inequality. We omit the proof as the argument is the same as in \cite{DHM}.

\begin{lemma}\label{LGL-loc} 
Consider the equation ${\mathcal Lu}=0$ with coefficients satisfying assumptions of Theorem~\ref{S3:T1}. Consider any boundary ball $\Delta_d=\Delta_d(Q)\subset {\mathbb R}^{n-1}$, let $A_d=(d/2,Q)$ be its corkscrew point and let
\begin{equation}
\nu_0=\left(\dint_{B_{d/4}(A_d)}|\nabla u(z)|^2\,dz\right)^{1/2}.
\end{equation}
Then for each $\gamma\in(0,1)$ there exists a constant $C(\gamma)>0$ 
such that $C(\gamma)\to 0$ as $\gamma\to 0$ and with the property that for each $\nu>2\nu_0$ and 
each energy solution $u$ of ${\mathcal Lu}=0$ there holds 
\begin{align}\label{eq:gl2}
&\hskip -0.20in 
\Big|\Big\{x'\in {\BBR}^{n-1}:\,\tilde{N}_a(\nabla u\chi_{T(\Delta_d)})>\nu,\,(M(S^2_b(\nabla u)))^{1/2}\leq\gamma\nu,
\nonumber\\[4pt] 
&\hskip 0in
\big(M(S^2_b(\nabla u))M(\tilde{N}_a^2(\nabla u\chi_{T(\Delta_d)}))\big)^{1/4}\leq\gamma\nu,\,
\, (M(\|\mu\|_{\mathcal C}^{1/2}\tilde{N}_a^2(\nabla u\chi_{T(\Delta_d)}))^{1/2}\le\gamma\nu\Big\}\Big|
\nonumber\\[4pt] 
&\hskip 0.50in
\quad\le C(\gamma)\left|\big\{x'\in{\BBR}^{n-1}:\,\tilde{N}_a(\nabla u\chi_{T(\Delta_d)})(x')>\nu/32\big\}\right|.
\end{align}
Here $\chi_{T(\Delta_d)}$ is the indicator function of the Carleson region $T(\Delta_d)$ and the square function
$S_b$ in \eqref{eq:gl2} is truncated at the height $2d$. Similarly, the Hardy-Littlewood maximal operator $M$
is only considered over all balls $\Delta'\subset\Delta_{md}$ for some enlargement constant $m=m(a)\ge 2$.
\end{lemma}

Finally we have the following.

\begin{proposition}\label{S3:C7} Consider the equation ${\mathcal Lu}=0$ in $\Omega=\BBR^n_{+}$ with coefficients satisfying assumptions of Theorem~\ref{S3:T1}. The for any $p>0$
and $a>0$ there exists an integer $m=m(a)\ge 2$ and finite constants $K=K(n,\lambda,\Lambda,p,a)>0$, $C=C(n,\lambda,\Lambda,p,a)>0$ such that if
$$\|\mu\|_{\mathcal C}< K,$$
then for all balls $\Delta_d\subset{\mathbb R}^{n-1}$ we have
\begin{equation}\label{S3:C7:E00ooloc}
\|\tilde{N}^r_a(\nabla u)\|_{L^{p}(\Delta_d)}\le C\|S^{2r}_a(\nabla u)\|_{L^{p}(\Delta_{md})}+Cd^{(n-1)/p}|\widetilde{\nabla u}(A_d)|,
\end{equation}
where $A_d$ denotes the corkscrew point of the ball $\Delta_d$ and $\widetilde{\nabla u}$ is as in \eqref{Eqqq-26}.

We also have a global estimate for any $p>0$ and $a>0$. Under the same assumptions as above (and extra a priori assumption $\|\tilde{N}_a(\nabla u)\|_{L^{p}({\BBR}^{n-1})}<\infty$ when $p< 2$)  there exists a finite 
constant $C=C(n,\lambda,\Lambda,p,a)>0$ such that 
\begin{equation}\label{S3:C7:E00oo}
\|\tilde{N}_a(\nabla u)\|_{L^{p}({\BBR}^{n-1})}\le C\|S_a(\nabla u)\|_{L^{p}({\BBR}^{n-1})}.
\end{equation}
\end{proposition}

\begin{proof} When $p>2$ \eqref{S3:C7:E00ooloc} follows immediately by a standard argument (multiplying the good-$\lambda$ inequality  \eqref{eq:gl2} by $\nu^{p-1}$ and integrating in $\nu$ over the interval $(2\nu_0,\infty)$). Note that the fact that the square function $S^{2r}_a$ is only integrated over some enlargement of $\Delta_d$ instead of the whole ${\mathbb R}^{n-1}$ follows from the fact that the set $\{x'\in{\BBR}^{n-1}:\,\tilde{N}_a(\nabla u\chi_{T(\Delta_d)})(x')>\nu/32\big\}$ on the righthand side of \eqref{eq:gl2} vanishes outside a ball of diameter comparable to $\Delta_d$. For this reason the maximal operators $M$ in \eqref{eq:gl2} can be restricted to such enlarged ball $\Delta_{md}$. 

The condition $\|\mu\|_{\mathcal C}< K$ comes from the presence of the term\newline $(M(\|\mu\|_{\mathcal C}^{1/2}\tilde{N}^2_a(\nabla u\chi_{T(\Delta_d)})))^{1/2}\le\gamma\nu$ in the good-$\lambda$ inequality. The argument that shows  \eqref{S3:C7:E00ooloc} for all $p>0$ can be found in \cite{FSt}.  The local estimate \eqref{S3:C7:E00ooloc} for $p>2$ is the necessary ingredient for what is otherwise a purely real variable argument. Further details can be found in  \cite{FSt}. 

Finally taking the limit $d\to\infty$ yields \eqref{S3:C7:E00oo}. The additional assumption $\|\tilde{N}_a(\nabla u)\|_{L^{p}({\BBR}^{n-1})}<\infty$ when $p< 2$ comes into play it order to guarantee that the term $d^{(n-1)/p}|\widetilde{\nabla u}(A_d)|$ in 
\eqref{S3:C7:E00ooloc} converges to zero as $d\to\infty$.
\end{proof}

\section{Estimates for the $p$-adapted square function.}
\label{S4}

Let $\Omega=\BBR^n_+$ and assume $u$ is a weak solution ${\mathcal L}u=0$ where
\begin{equation} \mathcal Lu = \partial_{i}\left(A_{ij}(x)\partial_{j}u\right) 
+B_{i}(x)\partial_{i}u\label{eq-zoncolan}
\end{equation}
with the Dirichlet boundary datum $f\in \dot{B}^{2,2}_{1/2}(\partial\Omega;{\BBC}) \cap  \dot{W}^{1,p}(\partial \Omega;{\BBC})$. Assume that 
$A$ is $p$-elliptic and smooth in $\BBR^n_+$ with $A_{00} =1$ and $A_{0j}$ real and that the measure $\mu$ defined as in \eqref{Car_hatAA} is Carleson.

Fix an arbitrary $y'\in\partial\Omega\equiv{\mathbb{R}}^{n-1}$ and consider $\Delta=\Delta_r(y)$; a ball of radius $r$ in ${\mathbb{R}}^{n-1}$ centred at $y'$. Pick 
a smooth cutoff function $\zeta$ which is $x_0-$independent and satisfies
\begin{equation}\label{cutoff-F}
\zeta= 
\begin{cases}
1 & \text{ in } \Delta, 
\\
0 & \text{ outside } 2\Delta,
\end{cases}
\end{equation}
where $2\Delta$ is a ball of radius $2r$ centered at $y'$.
Moreover, assume that $r|\nabla \zeta| \leq c$ for some positive constant $c$ independent of $y'$. 
We note that since
$$\partial_0(A_{0j}\partial_ju)=\partial_j(A_{0j}\partial_0u)-(\partial_jA_{0j})\partial_0u+(\partial_0A_{0j})\partial_ju,$$
we may as well assume that $A_{0j}=0$, $j>0$ by changing coefficients $A_{0j}$ and $A_{j0}$ of the matrix $A$ and modifying $B$. We note that this does not affect ellipticity of $A$ as all $A_{0j}$ are assumed to be real. It follows that, we can assume $\partial_k A_{0j}=0$ for all $j,k=0,1,\dots,n-1$.

We begin by considering the integral quantity for some function $w$ (to be specified later) such that $w|w|^{p/2-1}\in W^{1,2}_{loc}(\Omega)$
\begin{equation}\label{A00}
\mathcal{I}:=\mathscr{R}e\,\iint_{[0,s]\times 2\Delta}A_{ij}\partial_{j}w 
\partial_{i}(|w|^{p-2}\overline{w})x_0\zeta\,dx'\,dx_0
\end{equation}
with the usual summation convention understood. Here $s\in [0,r]$. With $\chi=x_0\zeta$ we have by Theorem 2.4 of \cite{DPcplx} for all $p$ for which $A$ is $p$-elliptic
for some $\lambda_p>0$
\begin{equation}\label{cutoff-AA}
\mathcal{I}\geq{\lambda_p}\iint_{[0,s]\times 2\Delta}|w|^{p-2}|\nabla w|^2 x_0\zeta\,dx'\,dx_0.
\end{equation}

The objective is to ultimately apply (\ref{cutoff-AA}) to $w=\partial_i u$, $i=1,...n-1$ and obtain a quantity that can be bounded from above by expressions that involve $L^p$ norms of 
$|\nabla u|$, and nontangential maximal functions of $|\nabla u|$, on the boundary. To see this, we continue the calculation using the fact that we can bound the right hand side of  (\ref{cutoff-AA})  by the expression 
$\mathcal{I}$ which brings in the equation. 
For the moment, we ignore the issue of finiteness of this expression, even though we use this fact in the calculations that follow. We'll return to this point after some of the basic calculations, for the sake of clarity of exposition.

The idea now is to integrate by parts the formula for $\mathcal I$ in order 
to relocate the $\partial_i$ derivative. This gives 
\begin{align}\label{I+...+IV}
\mathcal{I}
&= \mathscr{R}e\,\int_{\partial\left[(0,s)\times 2\Delta\right]} 
A_{ij}\partial_{j}w|w|^{p-2}\overline{w}x_0\zeta\nu_{x_i}\,d\sigma 
\nonumber\\[4pt]
&\quad - \mathscr{R}e\,\iint_{[0,s]\times 2\Delta}\partial_{i}\left(A_{ij} 
\partial_{j}w\right)|w|^{p-2}\overline{w}x_0\zeta\,dx'\,dx_0 
\nonumber\\[4pt]
&\quad - \mathscr{R}e\,\iint_{[0,s]\times 2\Delta}A_{ij}\partial_{j}{w}|w|^{p-2}\overline{w}\partial_{i}x_0\zeta\,dx'\,dx_0 
\nonumber\\[4pt]
&\quad - \mathscr{R}e\,\iint_{[0,s]\times 2\Delta}A_{ij}\partial_{j}w|w|^{p-2}\overline{w}x_0\partial_{i}\zeta\,dx'\,dx_0
\nonumber\\[4pt]
&=:I+II+III+IV,
\end{align}
where $\nu$ is the outer unit normal vector to $(0,s)\times 2\Delta$. The boundary term $I$ does not vanish
only on the set $\{s\}\times 2\Delta$ and only when $i=0$. This gives
\begin{equation}\label{cutoff-BBB}
I=\mathscr{R}e\,\int_{\{s\}\times 2\Delta} 
A_{0j}\partial_{j}w|w|^{p-2}\overline{w}x_0\zeta\,d\sigma 
\end{equation}

As $\partial_ix_0=0$ for $i>0$ the term $III$ is non-vanishing only for $i=0$. Since $A_{0j}=0$ for $j>0$ and $A_{00} =1$ term $III$   simplifies to

\begin{align}\label{u6fF}
III &=-\mathscr{R}e\,\iint_{[0,s]\times 2\Delta}\partial_{0}{w}|w|^{p-2}\overline{w}\zeta\,dx'\,dx_0 \nonumber\\
&=-\frac1p\iint_{[0,s]\times 2\Delta}  \partial_{0}(|w|^{p})\zeta\,dx'\,dx_0 \\
&=-\frac{1}{p}\int_{2\Delta} |w|^p(s,x')\zeta\,dx' + \frac{1}{p}\int_{2\Delta} |w|^p(0,x')\zeta\,dx' \nonumber
\end{align}

We add up all terms we have so far to obtain
\begin{equation}\label{square01}\begin{split}
\mathcal I
&\leq p^{-1}\int_{2\Delta}  \partial_{0}(|w|^p)(s,x') s \zeta \,dx' - \mathscr{R}e\,\iint_{[0,s]\times 2\Delta}\partial_{i}\left(A_{ij} 
\partial_{j}w\right)|w|^{p-2}\overline{w}x_0\zeta\,dx'\,dx_0 \\
&\quad  - {p}^{-1}\int_{2\Delta} |w|^p(s,x')\zeta\,dx' + {p}^{-1}\int_{2\Delta} |w|^p(0,x')\zeta\,dx'  + IV.
\end{split}\end{equation}

So far $w$ was an arbitrary function. We now apply \eqref{square01} to $w_k=\partial_k u$, $k=1,2,\dots,n-1$ where $u$ solves ${\mathcal L}u=0$. It follows that each $w_k$ solves
\begin{equation}\label{eq-partial}
{\mathcal L}w_k=\partial_i(A_{ij}\partial_jw_k)+B_iw_k=\partial_i((\partial_k A_{ij})w_j)-\partial_k(B_i)w_i.
\end{equation}
It follows that
\begin{equation}\label{eq-II}\begin{split}
II&=- \mathscr{R}e\,\iint_{[0,s]\times 2\Delta}\partial_{i}\left(A_{ij} 
\partial_{j}w_k\right)|w_k|^{p-2}\overline{w_k}x_0\zeta\,dx'\,dx_0 \\
&=\mathscr{R}e\,\iint_{[0,s]\times 2\Delta}B_iw_k|w_k|^{p-2}\overline{w_k}x_0\zeta\,dx'\,dx_0\\
&+\mathscr{R}e\,\iint_{[0,s]\times 2\Delta}(\partial_iA_{ij})w_j|w_k|^{p-2}\overline{w_k}x_0\partial_k\zeta\,dx'\,dx_0\\
&-\mathscr{R}e\,\iint_{[0,s]\times 2\Delta}B_iw_i|w_k|^{p-2}\overline{w_k}x_0\partial_k\zeta\,dx'\,dx_0\\
&+\mathscr{R}e\,\iint_{[0,s]\times 2\Delta}(\partial_iA_{ij})w_j\partial_k(|w_k|^{p-2}\overline{w_k})x_0\zeta\,dx'\,dx_0\\
&-\mathscr{R}e\,\iint_{[0,s]\times 2\Delta}(\partial_kA_{ij})(\partial_iw_j)|w_k|^{p-2}\overline{w_k}x_0\zeta\,dx'\,dx_0\\
&-\mathscr{R}e\,\iint_{[0,s]\times 2\Delta}B_i\partial_k(w_i|w_k|^{p-2}\overline{w_k})x_0\zeta\,dx'\,dx_0.
\end{split}\end{equation}
Here we integrated by parts terms containing two derivatives of $A$ or one derivative of $B$ by moving $\partial_k$ derivative. It is important here that $k\ne 0$ and hence $\partial_kx_0=0$. 
The first term on the righthand side 
can be estimated directly using Theorem~\ref{T:Car} while the last three terms we estimate using Cauchy-Schwarz inequality, the Carleson conditions for $A$ and $B$ and 
Theorem~\ref{T:Car}
\begin{align}\label{TWO-TWO}
|II_4|+|II_{5}|+|II_6| &\leq\left(\iint_{[0,s]\times 2\Delta}\left(|\nabla A|+|B|\right)^{2} 
|w|^{p} x_0\zeta\,dx'\,dx_0\right)^{1/2}  
\cdot\nonumber\\
&\qquad\left(\iint_{[0,s]\times 2\Delta}|w|^{p-2}|\nabla w|^{2}x_0\zeta\,dx'\,dx_0\right)^{1/2} 
\nonumber\\[4pt]
&\leq C(\lambda_p,\Lambda,p,n)\left(\|\mu\|_{\mathcal{C}}\int_{2\Delta} 
\left[\tilde{N}^r_{p,a}(w)\right]^{p}\,dx'\right)^{1/2}\cdot\mathcal{I}^{1/2}. 
\end{align}
In summary we get (after using AG-inequality)
$$|II|\le C(\lambda_p,\Lambda,p,n)\|\mu\|_{\mathcal{C}}\int_{2\Delta} 
\left[\tilde{N}^r_{p,a}(\nabla u)\right]^{p}\,dx' +\frac12\mathcal{I}+II_{2}+II_3.$$
It follows that \eqref{square01} simplifies to (after summing over $k=1,2,\dots n-1$)
\begin{equation}\label{square01aa}\begin{split}
\sum_{k=1}^{n-1}\mathcal I_k
&\leq p^{-1}\int_{2\Delta}  \partial_{0}(|\nabla_T u|^p)(s,x') r \zeta \,dx'  \\
&\hskip-1cm  - {p}^{-1}\int_{2\Delta} |\nabla_T u|^p(s,x')\zeta\,dx' + {p}^{-1}\int_{2\Delta} |\nabla_T u|^p(0,x')\zeta\,dx'  \\
&\hskip-1cm+ C(\lambda_p,\Lambda,p,n) \|\mu\|_{\mathcal{C}} \int_{2\Delta} \left[\tilde{N}^{r}_{p,a}(\nabla u)\right]^p \,dx' +II_{2}+II_3  + IV.
\end{split}\end{equation}

We estimate the terms $IV$. It can be bounded (up to a constant) by
\begin{equation}\label{eq5.16}
\iint_{[0,s]\times 2\Delta}|\nabla w||w|^{p-1}x_0|\partial_T\zeta|dx'dx_0,
\end{equation}
where $\partial_T\zeta$ denotes any of the derivatives in the direction parallel to the boundary. Recall that $\zeta$ is a smooth cutoff function equal to $1$ on $\Delta$ and $0$ outside $2\Delta$. In particular, we may assume $\zeta$ to be of the form $\zeta=\eta^2$ for another smooth function $\eta$ such that $|\nabla_T\eta|\le C/r$. By Cauchy-Schwarz \eqref{eq5.16} can be further estimated by
\begin{equation}\label{eq5.17}
\left(\iint_{[0,s]\times 2\Delta}|\nabla w|^2|w|^{p-2}x_0(\eta)^2dx'dx_0\right)^{1/2}\left(\iint_{[0,s]\times 2\Delta}|w|^{p}x_0|\nabla_T\eta|^2dx'dx_0\right)^{1/2}
\end{equation}
\begin{equation}
\lesssim{\mathcal I}^{1/2}\left(\frac1r\iint_{[0,s]\times (2\Delta\setminus\Delta)}|w|^pdx'dx_0\right)^{1/2}\le \varepsilon{\mathcal I}+C_\varepsilon\int_{2\Delta\setminus\Delta}\left[\tilde{N}^r_{p,a}(\nabla u)\right]^{p}\,dx'.\nonumber
\end{equation}
In the last step we have used AG-inequality and a trivial estimate of the solid integral $|u|^p$ by the $p$-averaged nontangential maximal function. 

Terms $II_2$ and $II_3$ are also similar. We use $|\nabla A|,|B|\le \|\mu|^{1/2}_{\mathcal C}/x_0$ and what remains has a trivial estimate by $\int_{2\Delta\setminus\Delta}\left[\tilde{N}^r_{p,a}(\nabla u)\right]^{p}\,dx$.
Substituting this and \eqref{eq5.17} into \eqref{square01aa} and by integrating in $s$ over $[0,r]$ and dividing by $r$ we finally obtain
\begin{align}\label{square02aa-loc}
&\hskip -0.20in 
\iint_{\Delta}\left[S^{r/2}_p(\nabla_T u)\right]^p\,dx'\le
\nonumber\\[4pt]
&\hskip -0.20in 
2\sum_{k=1}^{n-1}\iint_{[0,r]\times\Delta}|\nabla(\partial_k u)|^2|\partial_ku|^{p-2}\,x_0(1-{\textstyle\frac{x_0}{r}})\,dx'\,dx_0\lesssim
\nonumber\\[4pt]
&\quad  + {p}^{-1}\int_{2\Delta} |\nabla_T u|^p(0,x')\,dx' + {p}^{-1}\int_{2\Delta} |\nabla_T u|^p(r,x')\,dx'  
\nonumber\\[4pt]
&\quad +C\|\mu\|_{\mathcal{C}}\int_{2\Delta}\left[\tilde{N}^{r}_{p,a}(\nabla u)\right]^p\,dx'+C\int_{2\Delta\setminus\Delta}\left[\tilde{N}^r_{p,a}(\nabla u)\right]^{p}\,dx'.
\end{align}
We return now to the issue of finiteness of the quantities in \ref{cutoff-AA}.
We fix an $\varepsilon > 0$ and consider a bound for the 
expression
\begin{equation}\label{eps}
\iint_{[\varepsilon,s]\times 2\Delta}|w|^{p-2}|\nabla w|^2 (x_0-\varepsilon)\zeta\,dx'\,dx_0
\end{equation}
instead of $\iint_{[0,s]\times 2\Delta}|w|^{p-2}|\nabla w|^2 x_0\zeta\,dx'\,dx_0$.
The quantity \ref{eps} is finite by the interior estimates \eqref{RHthm2}.
By Theorem \ref{T:Car}, all of the previous calculations for the term (\ref{eps}), after averaging in $s$ will give the following estimate:
\begin{align}
&\hskip -0.20in 
\sum_{k=1}^{n-1}\iint_{[\varepsilon,r/2]\times\Delta}|\nabla(\partial_k u)|^2|\partial_ku|^{p-2}\,(x_0-\varepsilon))\,dx'\,dx_0\lesssim
\nonumber\\[4pt]
&\quad  + {p}^{-1}\int_{2\Delta} |\nabla_T u|^p(\varepsilon,x')\,dx' + {p}^{-1}\int_{2\Delta} |\nabla_T u|^p(r,x')\,dx'  
\nonumber\\[4pt]
&\quad +C\|\mu\|_{\mathcal{C}}\int_{2\Delta}\left[\tilde{N}^{r}_{p,a,\varepsilon}(\nabla u)\right]^p\,dx'+C\int_{2\Delta\setminus\Delta}
\left[\tilde{N}^r_{p,a,\varepsilon}(\nabla u)\right]^{p}\,dx'.
\end{align}
where $\tilde{N}^r_{p,a,\varepsilon}(\nabla u)$ denotes the nontangential maximal function relative to the domain $\{x_0 > \varepsilon\}$ as 
defined in \eqref{TFC-6x}.

To deal with the quantity $\int_{2\Delta} |\nabla_T u|^p(\varepsilon,x')\,dx' $, which is not itself a priori finite, we average the inequalities above 
over $\varepsilon \in [\varepsilon_0/2, \varepsilon_0]$. Averaging in $r$ as we have done earlier bounds the boundary integral
$\int_{2\Delta} |\nabla_T u|^p(r,x')\,dx'$ by a solid integral and   
we obtain, for each $k = 1,...n-1$,
\begin{equation}
\iint_{[\varepsilon_0,r/4]\times\Delta}|\nabla(\partial_k u)|^2|\partial_ku|^{p-2}\,(x_0-\varepsilon_0)\,dx'\,dx_0
\lesssim
C\int_{2\Delta}
\left[\tilde{N}^r_{p,a,\varepsilon_0/2}(\nabla u)\right]^{p}\,dx'.
\end{equation}
\hskip -0.20in

By Fatou's lemma, letting $\varepsilon_0 \rightarrow 0$,  the expressions in (\ref{cutoff-AA}) have an upper bound in terms of 
$\int_{2\Delta}\left[\tilde{N}^{r}_{p,a}(\nabla u)\right]^p\,dx$. Whenever this nontangential maximal function expression is
finite, the calculations leading to (\ref{square02aa-loc}) that depend on the finiteness of (\ref{cutoff-AA}) are justified.

To obtain a global version of \eqref{square02aa-loc}, consider a sequence of disjoint boundary balls 
$(\Delta_r(y'_j))_{k\in\mathbb N}$ such that $\cup_{j}\Delta_{2r}(y'_j) $ covers $\partial\Omega={\BBR}^{n-1}$ 
and consider a partition of unity $(\zeta_{j})_{k\in\mathbb N}$ subordinate to this cover. That is, 
assume $\sum_j \zeta_{j} = 1$ on ${\BBR}^{n-1}$ and each $\zeta_{j}$ is supported in $\Delta_{2r}(y'_j)$. 
Given that $\sum_j \partial_i\zeta_{j} = 0$ for each $i$, it follows by summing over all $k$
that 
$$\sum_{j} II_{2}+II_3 + IV= 0.$$
It follows from \eqref{square01aa} (after averaging in $s$ over $[0,r]$)
\begin{align}\label{square02aa}
&\hskip -0.20in 
\iint_{{\BBR}^{n-1}}\left[S^{r/2}_p(\nabla_T u)\right]^p\,dx'\le
\nonumber\\[4pt]
&\hskip -0.20in 
2\sum_{k=1}^{n-1}\iint_{[0,r]\times{\BBR}^{n-1}}|\nabla(\partial_k u)|^2|\partial_ku|^{p-2}\,x_0(1-{\textstyle\frac{x_0}{r}})\,dx'\,dx_0\lesssim
\nonumber\\[4pt]
&\quad  + {p}^{-1}\int_{{\BBR}^{n-1}} |\nabla_T u|^p(0,x')\,dx' + {p}^{-1}\int_{{\BBR}^{n-1}} |\nabla_T u|^p(r,x')\,dx'  
\nonumber\\[4pt]
&\quad +C\|\mu\|_{\mathcal{C}}\int_{{\BBR}^{n-1}}\left[\tilde{N}^{r}_{p,a}(\nabla u)\right]^p\,dx'.
\end{align}

We now modify our calculation above by considering a Lipschitz function $g:{\mathbb R}^{n-1}\to [0,\infty)$ such that $\sup_{2\Delta} g \le r/4$ (we only assume this to avoid integration over an empty set). We perform the same calculation starting from \eqref{A00} but this time we integrate over the set
$$([0,s]\times 2\Delta)\cap\Omega_g,$$
where 
$$\Omega_g:=\{(x_0,x')\in\mathbb R\times{\mathbb R}^{n-1}:\, x_0>g(x')\}.$$

Rather that repeating the whole calculation again we focus on the differences. We note that 
we will only consider $s\in[r/2,2r]$ to avoid complications that might arise from integration over empty sets.

The first difference will be that the term $I$ of \eqref{I+...+IV} will contain an additional boundary and hence
\begin{align}\label{cutoff-BBBv2}
I&=\mathscr{R}e\,\int_{\{s\}\times 2\Delta} 
(\partial_{0}w)|w|^{p-2}\overline{w}x_0\zeta\,d\sigma\\&+ \mathscr{R}e\,\int_{([0,s]\times2\Delta)\cap\partial\Omega_g} 
A_{ij}\partial_{j}w|w|^{p-2}\overline{w}x_0\zeta\nu_i\,d\sigma,\nonumber
\end{align}
where $\nu_i$ is the $i$-component of the outer normal of $\partial\Omega_g$.
Term \eqref{u6fF} becomes
\begin{align}\label{u6fFv2}
III 
&=-\frac{1}{p}\int_{2\Delta} |w|^p(s,x')\zeta\,dx' + \frac{1}{p}\int_{2\Delta} |w|^p(g(x'),x')\zeta\,dx' .
\end{align}
We look at the term $II$. As we integrate by parts to obtain \eqref{eq-II} we pick up two extra boundary terms.
\begin{align}\label{eq-IIbd}
II_{bd} =&-
\mathscr{R}e\,\int_{([0,s]\times 2\Delta)\cap\partial\Omega_g}(\partial_iA_{ij})w_j|w_k|^{p-2}\overline{w_k}x_0\nu_k\zeta\,d\sigma\\
&+\mathscr{R}e\,\int_{([0,s]\times 2\Delta)\cap\partial\Omega_g}B_iw_i|w_k|^{p-2}\overline{w_k}x_0\nu_k\zeta\,d\sigma.
\nonumber
\end{align}

We also modify some estimates. Terms $II_5$, $II_6$ and $II_7$ of \eqref{eq-II} are now integrated over the set
$([0,s]\times 2\Delta)\cap\Omega_g$ which allow us to use the estimate \eqref{Ca-222-x} of Theorem~\ref{T:Car}. This gives us
\begin{equation}
|II_5|+|II_6|+|II_7|\lesssim\left(\|\mu\|_{\mathcal{C}}\int_{([0,s]\times 2\Delta)\cap\partial\Omega_g} 
\left[\tilde{N}^{2r}_{p,a,g}(\nabla u)\right]^{p}\,dx'\right)^{1/2}\cdot\mathcal{I}^{1/2}. 
\end{equation}

Similar observation applies to terms $II_2$, $II_3$ and $IV$. It follows that what we have so far implies the estimate for some $c_p>0$:
\begin{equation}\label{square01aav2}\begin{split}
&\hskip-0.5cmc_p\sum_{k=1}^{n-1}\mathcal \iint_{([0,s]\times 2\Delta)\cap\Omega_g}|\nabla(\partial_k u)|^2|\partial_k u|^{p-2}x_0\zeta\,dx'dx_0\\
&\hskip-1cm \leq\, p^{-1}\int_{2\Delta}  \partial_{0}(|\nabla_T u|^p)(s,x') r \zeta \,dx'  \\
&\hskip-1cm  - {p}^{-1}\int_{2\Delta} |\nabla_T u|^p(s,x')\zeta\,dx' + {p}^{-1}\int_{2\Delta} |\nabla_T u|^p(g(x'),x')\zeta\,dx'  \\
&\hskip-1cm+ C(\lambda_p,\Lambda,\|\mu\|_{\mathcal{C}},p,n) \int_{T(2\Delta)\times\partial\Omega_g} \left[\tilde{N}^{2r}_{p,a,g}(\nabla u)\right]^p \,dx'\\
&\hskip-1cm+ \sum_{k=1}^{n-1}\mathscr{R}e\,\int_{([0,s]\times2\Delta)\cap\partial\Omega_g} 
A_{ij}\partial_{j}(\partial_k u)|\partial_k u|^{p-2}\overline{\partial_k u}\,x_0\zeta\nu_i\,d\sigma+II_{bd}.
\end{split}\end{equation}
Our goal is to estimate the first two terms on the right-hand side of \eqref{square01aav2} by 
$\tilde{N}^{2r}_{p,a,g}$. To do that we average in $s$ twice. We first integrate \eqref{square01aav2}
over an interval $s\in[r/2(1+\theta), r(1+\theta)]$ and then integrate the resulting inequality again in $\theta\in[0,1]$. This turns both mentioned terms into solid integrals of $|\nabla_T u|^p$ over a Whitney-type box inside $\Omega_g$. This simplifies  \eqref{square01aav2} to
\begin{equation}\label{square01aav3}\begin{split}
&\hskip-0.5cm c_p\sum_{k=1}^{n-1}\mathcal \iint_{([0,r/2]\times \Delta)\cap\Omega_g}|\nabla(\partial_k u)|^2|\partial_k u|^{p-2}x_0\,dx'dx_0\\
&\hskip-1cm\le {p}^{-1}\int_{2\Delta} |\nabla_T u|^p(g(x'),x')\,dx'\\
&\hskip-1cm+  C(\lambda_p,\Lambda,\|\mu\|_{\mathcal{C}},p,n) \int_{T(2\Delta)\times\partial\Omega_g} \left[\tilde{N}^{2r}_{p,a,g}(\nabla u)\right]^p \,dx'\\
&\hskip-1cm+ \sum_{k=1}^{n-1}\mathscr{R}e\,\int_{([0,s]\times2\Delta)\cap\partial\Omega_g} 
A_{ij}\partial_{j}(\partial_k u)|\partial_k u|^{p-2}\overline{\partial_k u}\,x_0\zeta\nu_i\,d\sigma+II_{bd}.
\end{split}\end{equation}
We shall use \eqref{square01aav3} in the following Lemma.

\begin{lemma}\label{LGL2} 
Let $\Omega=\BBR^n_+$ and assume $u$ be the energy solution of \eqref{eq-zoncolan}.  Assume that 
$A$ is $p$-elliptic and smooth in $\BBR^n_+$ with $A_{00} =1$ and $A_{0j}$ real and that the measure $\mu$ defined as in \eqref{Car_hatAA} is Carleson.

Consider any $b>a>0$. Then for each $\gamma\in(0,1)$ there exists a constant $C(\gamma)>0$ 
such that $C(\gamma,a,b)\to 0$ as $\gamma\to 0$ and with the property that for each $\nu>0$ we have
\begin{align}\label{eq:gl2aa}
&\hskip -0.20in 
\left|\Big\{x'\in {\BBR}^{n-1}:\,S_{p,a}(\nabla_Tu)(x')>\nu,\, \tilde{N}_b(\nabla u)(x')\le\gamma \nu\Big\}\right|
\nonumber\\[4pt] 
&\hskip 0.50in
\quad\le C(\gamma)\left|\big\{x'\in{\BBR}^{n-1}:\,{S}_{p,b}(\nabla_Tu)(x')>\nu/2\big\}\right|.
\end{align}
Here $\tilde{N}_b$ denotes the $L^2$ version of the nontangential maximal function defined over cones of aperture $b$.
\end{lemma}

\begin{proof} 
We observe that $\tilde{N}_b(\nabla u)\le \gamma\nu$ also implies $\tilde{N}_{p,b}(\nabla u)\lesssim\gamma\nu$ thanks to Proposition \ref{Regularity}.  Also clearly, 
$\big\{x'\in{\BBR}^{n-1}:\,{S}_{p,b}(\nabla_T u)(x')>\nu/2\}$ is an open 
subset of ${\BBR}^{n-1}$. When this set is empty, or is all of ${\BBR}^{n-1}$, 
estimate \eqref{eq:gl2aa} is trivial, so we focus on the case when the set in question is 
both nonempty and proper. Granted this, we may consider a Whitney decomposition $(\Delta_i)_{i\in I}$ 
of it, consisting of open cubes in ${\mathbb{R}}^{n-1}$. Let $F_\nu^i$ be the set appearing on the 
left-hand side of \eqref{eq:gl2aa} intersected with $\Delta_i$.  Let $r_i$ be the diameter of $\Delta_i$. Due to the nature of the Whitney decomposition there exists a point $p'\in 2\Delta_i$ such that ${S}_{p,b}(\nabla_T u)(p')<\nu/2$.
From this and the fact that $b>a$ it follows that for all $x'\in F^i_\nu$ we have
$$S^{d}_{p,a}(\nabla_T u)(x')>\nu/2,$$
where $S^{d}_{p,a}$ is the truncated version of the square function at some height $d\approx r_i$, where the precise nature of relation between $d$ and $r_i$ depends on the apertures $a$ and $b$. 

For some $a<c<b$ consider the domain
$$\Omega_c=\bigcup_{x'\in F^i_\nu} \Gamma_c(x');$$
this is a Lipschitz domain with Lipschitz constant $1/c$. Observe that $F^i_\nu\subset \partial\Omega_c$.
It follows that
$$|F^i_\nu|\le \frac{2^p}{\nu^p}\int_{F^i_\nu}\left[S^{d}_{p,a}(\nabla_Tu)(x')\right]^p\,dx'\lesssim \nu^{-p}
\sum_{k=1}^{n-1}\mathcal \iint_{{\Omega_c\cap T(\Delta_i)}}|\nabla(\partial_k u)|^2|\partial_k u|^{p-2}x_0\,dx.$$
We apply \eqref{square01aav3}.  It follows that
\begin{align}\label{eq-zonc2}
|F^i_\nu|&\lesssim \nu^{-p}\Big\{\int_{\partial\Omega_c\cap T(2\Delta_i)}\left(\left|\nabla_T u\big|_{\partial\Omega_c}\right|^p+\left[\tilde{N}^{2d}_{p,a,c}(\nabla u)\right]^{p}\right)\,d\sigma\\
&+ \sum_{k=1}^{n-1}\Big[\mathscr{R}e\,\int_{T((2\Delta_i)\cap\partial\Omega_c} 
A_{ij}(\partial^2_{jk} u)|\partial_k u|^{p-2}\overline{\partial_k u}\,x_0\zeta\nu_i\,d\sigma\nonumber\\
&+\mathscr{R}e\,\int_{([0,s]\times 2\Delta)\cap\partial\Omega_g}(\partial_iA_{ij})\partial_j u| \partial_k u|^{p-2}\overline{\partial_k u}\,x_0\nu_k\zeta\,d\sigma\nonumber\\
&+\mathscr{R}e\,\int_{([0,s]\times 2\Delta)\cap\partial\Omega_g}B_i\partial_i u|\partial_k u|^{p-2}\overline{\partial_k u}\,x_0\nu_k\zeta\,d\sigma\Big]\Big\},
\nonumber
\end{align}
where $\tilde{N}^{2d}_{p,a,c}$ is defined using nontangential cones with aperture $a$ with vertices on $\partial\Omega_c$. Due to the fact that each of these cones is contained in one of the cones $\Gamma_b(x')$
for some $x'\in F^i_\nu$ (as $c<b$) and on $F^i_\nu$: $\tilde{N}_b(\nabla_Tu)(x')\le\gamma \nu$ we also have 
$\tilde{N}^{2d}_{p,a,c}(\nabla_Tu)\lesssim\gamma\nu$ everywhere on $\partial\Omega_c$.  This takes care of the second term. 

We still need to deal with the four other terms on the righthand side. We do it by converting these terms into a solid integrals by 
averaging $c$ over the interval $[a,(a+b)/2]$. Let us denote by
$$\mathcal O=\Omega_{(a+b)/2}\setminus\overline{\Omega_a}.$$
${\mathcal O}$ is the set over which the four terms we want to bound will integrate over after the averaging. The sets ${\Omega_c}$ also share $F^i_\nu$ as a common boundary, however
there we have a trivial estimate 
$$\int_{F^i_\nu}\left|\nabla_T u\big|_{\partial\Omega_c}\right|^pd\sigma\le (\gamma\nu)^p|\Delta_i|,$$
from the fact that $\tilde{N}_b(\nabla_Tu)(x')\le\gamma \nu$ on $F^i_\nu$, while the last three terms of
\eqref{eq-zonc2} vanish there (as $x_0=0$). 

Given the way the set $\mathcal O$ is defined geometric considerations imply that it can be covered by a non-overlapping collection of Whitney cubes $\{Q_j\}$ in ${\mathbb R}^n_+$ with the following properties:
\begin{equation}
{\mathcal O}\subset\bigcup_j Q_j,\qquad r_j=\mbox{diam}(Q_j)\approx\mbox{dist}(Q_i,\partial{\mathbb R}^n_+),\quad 2Q_j\subset \Omega_b.
\end{equation}

Furthermore the projections of $Q_j$ onto the boundary ${\mathbb R}^{n-1}$ are \lq\lq almost disjoint"; that is each such projection overlaps with at most $K=K(a,b)$ other projections. From this
$\sum_j \mbox{diam}(Q_j)^{n-1}\approx |2\Delta_i|$.

Consider the contribution of the first term on the right-hand side of \eqref{eq-zonc2} after the averaging  in $c$ on each $Q_j$. Such term can be bounded by
$$(\mbox{diam}(Q_j))^{-1}\iint_{Q_j}|\nabla u|^p dx\lesssim (\gamma\nu)^p \mbox{diam}(Q_j)^{n-1},$$
where the bound $\lesssim (\gamma\nu)^p$ comes from the fact that $Q_j\subset \Omega_b$ and hence the $L^p$ average of $\nabla u$ on $Q_j$ has this bound from our assumptions. Summing over all $Q_j$ gives us the bound
$$\sum_j(\mbox{diam}(Q_j))^{-1}\iint_{Q_j}|\nabla u|^p dx\lesssim (\gamma\nu)^p |2\Delta_i|.$$
In fact, we have this bound also for the fourth and fifth term on the right-hand side of \eqref{eq-zonc2} since $|\nu_k|\le 1$ and $|\nabla A|x_0, |B|x_0\lesssim\|\mu\|_{\mathcal C}^{1/2}$ and hence we are again dealing with a solid integral of $|\nabla u|^p$ over each $Q_j$. Finally, the third term on the right-hand side of \eqref{eq-zonc2} is somewhat different and on $Q_j$ has the bound by
$$(\mbox{diam}(Q_j))^{-1}\iint_{Q_j}|\partial_k u|^{p-1}|\nabla \partial_k u|x_0 dx$$
which since $x_0\approx \mbox{diam}(Q_j)$ is further bounded by Cauchy-Schwarz
$$\lesssim (\mbox{diam}(Q_j))^{-1}\left(\iint_{Q_j}|\partial_k u|^{p} dx\right)^{1/2}\left((\mbox{diam}(Q_j))^2\iint_{Q_j}|\nabla\partial_k u|^2|\partial_k u|^{p-2} dx\right)^{1/2}$$
where for the second term we can use \eqref{RHthm2} to again get bound of the whole expression by
$C\mbox{diam}(Q_j)^{n-1}$. It follows we have after the averaging procedure for every term of \eqref{eq-zonc2} the same bound (up to a constant)  and that 
$$|F^i_\nu|\le C(a,b,\|\mu\|_{\mathcal C})\gamma^p|\Delta_i|.$$
Summing over all $i$ yields \eqref{eq:gl2aa} as desired.
\end{proof}

We will require a localized version of Lemma \ref{LGL2} as well.

\begin{lemma}\label{LocalGoodL}
Let $u$ be as in Lemma \ref{LGL2}. Fix $R \geq h$ and consider a boundary ball $\Delta_R \subset {\BBR}^{n-1}$. Let $p\geq q > 1$ for any $q$ such that $A$ is $q$-elliptic.
Let 
$$\nu_0^p = C \dint_{\Delta_{2R}} \left[N_b^{2R}(\nabla u)\right]^p dx',$$
where $C$ is a constant depending only on dimension (calculated in the proof below).
Then for each $\gamma\in(0,1)$ there exists a constant $C(\gamma)>0$ 
such that $C(\gamma,a,b)\to 0$ as $\gamma\to 0$ and with the property that for each $\nu>\nu_0$ 
\begin{align}\label{eq:gl2-2}
&\hskip -0.20in 
\left|\Big\{x'\in \Delta_R :\,S^R_{q,a}(\nabla_Tu)(x')>\nu,\, \tilde{N}^{2R}_b(\nabla u)(x')\le\gamma \nu\Big\}\right|
\nonumber\\[4pt] 
&\hskip 0.50in
\quad\le C(\gamma)\left|\big\{x'\in \Delta_R:\,{S}_{q,b}(\nabla_Tu)(x')>\nu/2\big\}\right|.
\end{align}

\end{lemma} 
\begin{proof} It follows from \eqref{square02aa-loc} (by well-familiar averaging) that
\begin{equation}\label{eq-tourm}
\|S^{R}_{q, b}(\nabla_Tu)\|_{L^q(\Delta_R)} \lesssim \|N^{2R}_b(\nabla u)\|_{L^q(\Delta_{2R})}.
\end{equation}
Therefore,

\begin{align}\label{Sqwithnu}
\hskip -0.20in 
\big| \Delta_R \cap \{S^R_q > \nu/2\} \big|    &\lesssim  \nu^{-q} \|N^{2R}_b(\nabla u)\|^q_{L^q(\Delta_{2R})}\\
&
\lesssim \nu^{-q} \|N^{2R}_b(\nabla u)\|^{q/p}_{L^p(\Delta_{2R})} \big| \Delta_{2R} \big|^{1-q/p}  \nonumber\\
&\lesssim    C_{\varepsilon}\nu^{-p} \int_{\Delta_{2R}} (N^{2R}_b(\nabla u))^p + \varepsilon \big|\Delta_{R}\big|.
\end{align}
Choosing $\varepsilon = 1/4$, which determines $C_{\varepsilon}$,  and we now fix $C = 4C_{\varepsilon}$ in the definition of $\nu_0$.
This implies that for any $\nu>\nu_0$, we have that
$$\big| \Delta_R \cap \{S^R_{q,b} > \nu/2\} \big| < 1/2 \big| \Delta_R\big|.$$
Thus, there exists a Whitney decomposition of  $\Delta_R \cap \{S^R_{q,b} > \nu/2\}$ into open cubes $\Delta_i$ with the property that
$2\Delta_i \cap \Delta_R$ contains a point for which $S^R_{q,b}(\nabla_T u) < \nu/2.$
From this point on, the proof proceeds as in Lemma \ref{LGL2}.
\end{proof}

\begin{corollary}\label{S4:C1} Under the assumption of Lemma \ref{LGL2}, for any $q \geq p >1$ and $a>0$ there exists a finite 
constant $C=C(\lambda_p,\Lambda,p,q,a,\|\mu\|_{\mathcal C},n)>0$ such that 
\begin{equation}\label{S3:C7:E00oo=sdd}
\|S^R_{p,a}(\nabla_T u)\|_{L^{q}({\Delta_{R}})}  \le C\|\tilde{N}^{2R}_{p,a}(\nabla u)\|_{L^{q}({\Delta_{2R}})},
\end{equation}
\begin{equation}\label{S3:C7:E00oo=s}
\|S_{p,a}(\nabla_Tu)\|_{L^{q}({\BBR}^{n-1})}\le C\|\tilde{N}_{p,a}(\nabla u)\|_{L^{q}({\BBR}^{n-1})}.
\end{equation}
The inequality \eqref{S3:C7:E00oo=s} also holds for any $q > 0$, 
provided we know a priori that $\|S_{p,a}(\nabla_Tu)\|_{L^{q}({\BBR}^{n-1})}<\infty$.
\end{corollary}
\begin{proof} 

The estimate \eqref{S3:C7:E00oo=sdd}
is a consequence of the local good-$\lambda$ inequality established above and the equivalence (\cite{CMS}) of $p$-adapted square functions with different aperture in any $L^q$ norm.

When $q \geq p$, and $M$ is large,  
 $$ \int_0^M   \nu^{q-1}  \big| \Delta_R \cap \{S^R_{p,a}(\nabla_Tu) > \nu\}\big| d\nu \leq C(M) \int_0^M   \nu^{p-1}  \big| \Delta_R \cap \{S^R_{p,a}(\nabla_Tu) > \nu\}\big| d\nu.
 $$
 
 By \eqref{eq-tourm} and the fact that the coefficients are smooth, the right hand side is finite. Therefore, the left hand side is also bounded, with a constant that
 may depend on $M$. 
 
 Now we multiply the good-$\lambda$ inequality of
 Lemma \ref{LocalGoodL} by  $\nu^{p-1}$ and integrate separately over $(0, \nu_0)$ and $(\nu_0, M)$.
 This gives
 
$$ \|S^R_{p,a}(u)\|_{L^{q}({\Delta_{R}})}  \le C\|\tilde{N}^{2R}_{p,a}(u)\|_{L^{q}({\Delta_{2R}})}, $$

\noindent after taking the limit as $M \to \infty$. 

The estimate \eqref{S3:C7:E00oo=s} follows by taking the limit $R\to\infty$.

When $q<p$, the local good-$\lambda$ inequality is not available, which is why we need the additional a priori assumption $\|S_{p,a}(\nabla_Tu)\|_{L^{q}({\BBR}^{n-1})}<\infty$. The proof proceed otherwise as above but using Lemma \ref{LGL2}.
\end{proof}

So far we have avoided considering the square function of $\partial_0u$. We remedy it now. Observe
that since
$$|\nabla(\partial_0 u)|\le |\partial^2_{00}u|+|\nabla(\nabla_Tu)|,$$
we can use previous calculations for the square function of $\nabla_Tu$ and focus on $\partial^2_{00}u$.

Since $u$ solves ${\mathcal L}u=0$
and $A_{00}=1$ we have for
$$\partial^2_{00}u=-\sum_{(i,j)\ne (0,0)}\partial_i(A_{ij}\partial_ju)-\sum_iB_i\partial_i u.$$
It follows that we have the estimate:
\begin{equation}\label{zertex1}
S^R_{2,a}(\partial_0 u)(x')\le S^R_{2,a}(\nabla_T u)(x') +C\,{\mathcal T}^R_{a}(\nabla u)(x'),
\end{equation}
where we define
\begin{equation}\label{zertex2}
{\mathcal T}^R_{a}(\nabla u)(Q)=\left(\int_{\Gamma^R_a(Q)}(|\nabla A|^2+|B|^2)|\nabla u|^2\delta(x)^{2-n}dx\right)^{1/2},
\end{equation}
Considering the same $\Omega_g$ as above we have an analogue of \eqref{square01aav3}:

\begin{equation}\label{square01aav3-2}\begin{split}
&\hskip-1cm\iint_{([0,r]\times \Delta)\cap\Omega_g}(|\nabla A|^2+|B|^2)|\nabla u|^2x_0\,dx'dx_0\\
\hskip-1cm&\le C\|\mu\|_{\mathcal C}\int_{T(2\Delta)\times\partial\Omega_g} \left[\tilde{N}^{2r}_{p,a,g}(\nabla u)\right]^2 \,dx.
\end{split}\end{equation}

If follows we can establish a good-lambda inequality analogous to Lemma \ref{LGL2}.

\begin{lemma}\label{LGL} 
Let $\Omega=\BBR^n_+$ and assume $u$ be the energy solution of \eqref{eq-zoncolan}.  Assume that 
$A$ is $2$-elliptic and smooth in $\BBR^n_+$ with $A_{00} =1$ and $A_{0j}$ real and that the measure $\mu$ defined as in \eqref{Car_hatAA} is Carleson.

Consider any $b>a>0$. Then for each $\gamma\in(0,1)$ there exists a constant $C(\gamma)>0$ 
such that $C(\gamma,a,b)\to 0$ as $\gamma\to 0$ and with the property that for each $\nu>0$ we have
\begin{align}\label{eq:gl2v2}
&\hskip -0.20in 
\left|\Big\{x'\in {\BBR}^{n-1}:\,{\mathcal T}_{a}(\nabla u)(x')>\nu,\,\|\mu\|_{\mathcal{C}}^{1/2}\tilde{N}_b(\nabla u)(x')\le\gamma \nu\Big\}\right|
\nonumber\\[4pt] 
&\hskip 0.50in
\quad\le C(\gamma)\left|\big\{x'\in{\BBR}^{n-1}:\,{\mathcal T}_{b}(\nabla u)(x')>\nu/2\big\}\right|.
\end{align}
\end{lemma}

We omit the proof as it follows the same idea as the proof of Lemma \ref{LGL2} using \eqref{square01aav3-2} in place of \eqref{square01aav3}. Also averaging in $c$ is not needed. We also have an analogue of Lemma \ref{LocalGoodL} by the same argument. We record the consequences of these two results.

\begin{corollary}\label{S4:C2} Under the assumption of Lemma \ref{LGL2}, for any $q \geq 2$ and $a>0$ there exists a finite 
constant $C=C(\lambda_2,\Lambda,q,a,\|\mu\|_{\mathcal C},n)>0$ such that 
\begin{equation}\label{S3:C7:E00oo=sdd2}
\|S^R_{2,a}(\partial_0 u)\|_{L^{q}({\Delta_{R}})}  \le C\left[\|S^{R}_{2,a}(\nabla_T u)\|_{L^{q}(\Delta_{2R})}+\|\mu\|_{\mathcal{C}}^{1/2}\|\tilde{N}^{2R}_{2,a}(\nabla u)\|_{L^{q}({\Delta_{2R}})}\right],
\end{equation}
\begin{equation}\label{S3:C7:E00oo=s2}
\|S_{2,a}(\partial_0 u)\|_{L^{q}({\BBR}^{n-1})}\le C\left[\|S_{2,a}(\nabla_T u)\|_{L^{q}({\BBR}^{n-1})}+\|\mu\|_{\mathcal{C}}^{1/2}\|\tilde{N}_{2,a}(\nabla u)\|_{L^{q}({\BBR}^{n-1})}\right].
\end{equation}
The inequality \eqref{S3:C7:E00oo=s2} also holds for any $q > 0$, 
provided we know a priori that
\newline
$\|{\mathcal T}_{a}(\nabla u)\|_{L^{q}({\BBR}^{n-1})} < \infty.$
\end{corollary}

We are now ready to establish a local solvability result. Let us consider domains of the following the form. Let $\Delta_d\subset{\mathbb R}^{n-1}$
be a boundary ball or a cube of diameter $d$. We denote by ${\mathcal O}_{\Delta_d,a}$ 
\begin{equation}\label{Odom}
{\mathcal O}_{\Delta_d,a}=\bigcup_{x'\in\Delta_d}\Gamma_a(x').
\end{equation}
Here as before $\Gamma_a(x')$ denotes the nontangential region with aperture $a$ at a point $x'$ (c.f. Definition  \ref{DEF-1}).

Clearly,  a domain such as \eqref{Odom} is a domain with Lipschitz constant $1/a$. It follows that if ${\mathcal L}$ satisfies assumptions of this Theorem \ref{S3:T0} on ${\mathbb R}^{n}_+$ it also satisfies it on any domain ${\mathcal O}_{\Delta_d,a}$, provided $1/a$ is sufficiently small. This can be seen via the pullback transformation \eqref{E:rho} which transforms the problem from ${\mathcal O}_{\Delta_d,a}$ back to ${\mathbb R}^n_+$. This modifies the coefficients of our PDE to say
\begin{equation}\label{eq-pf14}
\mbox{div}(\bar{A}\nabla v)=0.
\end{equation}

In particular, if the original PDE on ${\mathcal O}_{\Delta_d,a}$ satisfies $A_{00}=1$ and $A_{0j}$ are real, the modified coefficients $\bar{A}$ will fail to do so. However, we could fix that via the change of coefficients discussed in \eqref{eqSWAP} together with the observations noted below. It follows that \eqref{eq-pf14} can be rewritten as
\begin{equation}\label{eq-pf15}
\mbox{div}(\tilde{A}\nabla v)+\tilde{B}\cdot \nabla v=0.
\end{equation}
Because $1/a$ is small the coefficient $\bar{A}_{00}$ is close to $1$ and $\bar{A}_{0j}$ are almost real. It follows that rewriting \eqref{eq-pf14} as \eqref{eq-pf15} will not destroy the ellipticity and $p$-ellipticity of the matrix $\tilde{A}$. Hence our previous results of this section apply as they were developed for operators ${\mathcal L}$ with first order (drift) terms. 

We note that in section \ref{SS:43}, drift terms are not allowed, but the results of this section can be applied to the PDE \eqref{eq-pf14} because the special assumptions on $\bar{A}_{0j}$ are not used there.

To discuss solvability on domain ${\mathcal O}_{\Delta_d,a}$ we need to consider the nontangential maximal function $\tilde N$ that is taken with respect to nontangential approach regions that are contained inside ${\mathcal O}_{\Delta_d,a}$; that is we need to take regions $\Gamma_b(\cdot)$ for any $b<a$. Without loss of generality we choose $b=a/2$ and fix it for the remaining part of this section. Finally, $\nabla_T u$ at the boundary of ${\mathcal O}_{\Delta_d,a}$ is understood to be the tangential component of the gradient with respect to the boundary of this domain.
  
For ease of notation we drop the dependence of the domain ${\mathcal O}_{\Delta_d,a}$ on $\Delta_d$ and $a$ and use ${\mathcal O}={\mathcal O}_{\Delta_d,a}$. We have the following result.

\begin{lemma}\label{S6:L1} Let $\mathcal L$ be as in Theorem \ref{S3:T0} on the domain ${\mathbb R}^n_{+}$ and let $A$ be $q$-elliptic for some $q\ge 2$.
Let $\mathcal O$ be a Lipschitz domain as above and assume $u$ is an arbitrary energy solution of ${\mathcal L}u=0$ in ${\mathbb R}^n_+$
with the Dirichlet boundary datum $\nabla_T f\in L^{q}(\partial \mathcal O;{\BBR}^N)$. Then there exists $m=m(a)>1$ and $K=K(\lambda_p,\Lambda,n,p)>0$ such that if
$$\|\mu\|_{\mathcal C}+a^{-1}<K,$$
the following estimate holds:
\begin{equation}\label{Main-Estlocxx}
\|\tilde{N}_{a/2} (\nabla u)\|_{L^{q}(\Delta_d)}\leq C_q\|\nabla_T f\|_{L^{q}(\partial{\mathcal O\cap \overline{T(\Delta_{md})}})}+C_qd^{(n-1)/q}\sup_{x\in \mathcal O\cap\{\delta(x)>d\}}W_2(x),
\end{equation}
where $\delta(x)=\mbox{dist}(x,\partial{\mathbb R}^n_+)$ and $W_2(x)=\left(\dint_{B_{\delta(x)/4}(x)}|\nabla u(y)|^2 dy)\right)^{1/2}$. 
\end{lemma}

\begin{proof} In last term of \eqref{Main-Estlocxx} because of the way $\mathcal O$ is defined we clearly have
\begin{equation}\label{brmbrm}
\{(x_0,x')\in\mathcal O:\, x'\notin\Delta_{(1+a)d}\}\subset  \mathcal O\cap\{\delta(x)>d\}.
\end{equation}
If follows that by considering the pull-back map $\rho:{\mathbb R}^n_+\to\mathcal O$ defined in \eqref{E:rho} proving \eqref{Main-Estlocxx} is equivalent to establishing
\begin{equation}\label{Main-Estlocxxy}
\|\tilde{N} (\nabla u)\|_{L^{q}(\Delta_d)}\leq C\|\nabla_T f\|_{L^{q}(\Delta_{md};{\BBR}^N)}+Cd^{(n-1)/q}\sup_{x\in {\mathbb R}^n_+\setminus T(\Delta_{(1+a)d})}W_2(x),
\end{equation}
where we now work on the domain ${\mathbb R}^n_+$ with $u$ solving $\mathcal Lu=0$ in ${\mathbb R}^n_+$ for $\mathcal L$ as in Theorem \ref{S3:T0}. We start with the term on the lefthand side of \eqref{Main-Estlocxxy}. If follows from \eqref{S3:C7:E00ooloc} that for some $m_1>1+a$

\begin{equation}\label{S3:C7:E00ooloc-2}
\|\tilde{N}^{(1+a)d}_a(\nabla u)\|^q_{L^{q}(\Delta_d)}\le C\|S^{m_1d}_a(\nabla u)\|^q_{L^{q}(\Delta_{m_1d})}+Cd^{n-1}|\widetilde{\nabla u}(A_d)|^q.
\end{equation}
The last term above has a trivial bound by $Cd^{n-1}\sup_{x\in {\mathbb R}^n_+\setminus T(\Delta_{(1+a)d})}[W_2(x)]^q$.

By \eqref{S3:C7:E00oo=sdd2} we have for $m_2=2m_1$:
\begin{equation}\nonumber
\|S^{m_1d}_{a}(\nabla u)\|_{L^{q}({\Delta_{m_1d}})}  \le C\left[\|S^{m_2d}_{a}(\nabla_T u)\|_{L^{q}(\Delta_{m_2d})}+\|\mu\|_{\mathcal{C}}^{1/2}\|\tilde{N}^{m_2d}_{a}(\nabla u)\|_{L^{q}({\Delta_{m_2d}})}\right].
\end{equation}

Using the H\"older inequality we have for any $x'\in{\mathbb R}^{n-1}$
\begin{align}
&\hskip-5mm\left[S_{2,a}(\nabla_T u)(x')\right]^2=\sum_{k>0}\iint_{\Gamma_a(x')}|\nabla\partial_k u|^{2/q}|\partial_k u|^{1-2/q}
|\nabla\partial_k u|^{2/q'}|\partial_k u|^{1-2/q'}x_0\,dx'\,dx_0\nonumber\\
&\le \sum_{k>0}\left(\iint_{\Gamma_a(x')}|\nabla\partial_k u|^2|\partial_k u|^{q-2}
x_0\,dx'\,dx_0\right)^{1/q}\times\nonumber\\
&\hskip3cm\left(\iint_{\Gamma_a(x')}|\nabla\partial_k u|^2|\partial_k u|^{q'-2}
x_0\,dx'\,dx_0\right)^{1/q'}\label{eq-pf3a}\\
\nonumber &\le S_{2,q}(\nabla_T u)(x')S_{2,q'}(\nabla_T u)(x').
\end{align}

Hence the previous line implies that for any $\varepsilon>0$ we have
\begin{equation}\label{eq-qq'}
S_{2,a}(\nabla_T u)(x')\le C_\varepsilon S_{q,a}(\nabla_T u)(x')+\varepsilon S_{q',a}(\nabla_T u)(x'),
\end{equation}
and the same inequality holds for the truncated square functions.
Observe that $q\ge q'$ and hence we can use  \eqref{S3:C7:E00oo=sdd} to estimate the second term.
This gives us
\begin{align}\nonumber
\|S^{m_2d}_a(\nabla_T u)\|^q_{L^q(\Delta_{m_2d})}\le C_\varepsilon\|S^{m_2d}_{q,a}(\nabla_T u)\|^q_{L^q(\Delta_{m_2d})}+
\varepsilon^q\|\tilde{N}^{m_3d}(\nabla u)\|^q_{L^{q}(\Delta_{m_3d})}
\end{align}

For some $m_3>m_2$. We choose $\varepsilon$ so that $\varepsilon^q= \|\mu\|^{q/2}_{\mathcal C}$. 
The estimates we have so far can be combined to the following estimate:
\begin{align}\label{eq-summaryest}
\|\tilde{N}^{(1+a)d}_a(\nabla u)\|^q_{L^{q}(\Delta_d)}&\le C\|S^{m_2d}_{q,a}(\nabla_T u)\|^q_{L^q(\Delta_{m_2d})}\\&\nonumber+C\|\mu\|_{\mathcal{C}}^{q/2}\|\tilde{N}^{m_3d}_{a}(\nabla u)\|^q_{L^{q}({\Delta_{m_3d}})}\\&\nonumber+
Cd^{n-1}\sup_{x\in {\mathbb R}^n_+\setminus T(\Delta_{(1+a)d})}[W_2(x)]^q.
\end{align}
To estimate the first term on the righthand side we use \eqref{square02aa-loc}. This gives
\begin{eqnarray}\label{altsim}
&&\|S^{m_2d}_{q,a}(\nabla_T u)\|^q_{L^q(\Delta_{m_2d})}\\
&\lesssim& 
\int_{\Delta_{m_3d}}|\nabla_T u(0,x')|^{q}\,dx' 
+\int_{\Delta_{m_3d}}|\nabla_T u(m_3d,x')|^{q}\,dx'\nonumber\\
&&\quad+\|\mu\|_{\mathcal{C}}\int_{\Delta_{m_3d}}\left[\tilde{N}^{m_3d}(\nabla u)\right]^{q}\,dx'+C\int_{\Delta_{m_3d}\setminus\Delta_{m_2d}}\left[\tilde{N}^{m_3d}(\nabla u)\right]^{p}\,dx'.\nonumber
\end{eqnarray}

Observe that if the estimate above holds for certain $m_3>1$ it will certainly holds for any larger value, say $2m_3$. Hence we can average the estimate on the righthand side of \eqref{square02aa-loc} between $m_3$ and $2m_3$. This turns the second term on the righthand side of  \eqref{square02aa-loc}
 into a solid integral over a set that is contained in ${\mathbb R}^n_+\setminus T(\Delta_{(1+a)d})$
and therefore bounded by $Cd^{n-1}\sup_{x\in {\mathbb R}^n_+\setminus T(\Delta_{(1+a)d})}[W_2(x)]^q$. Hence we have for $m_4=2m_3$ thanks to \eqref{eq-summaryest}:

\begin{eqnarray}\label{altsim3}
&&\|\tilde{N}^{(1+a)d}(\nabla u)\|^q_{L^{q}(\Delta_d)}\lesssim 
\int_{\Delta_{m_4d}}|\nabla_T f(x')|^{q}\,dx' \\
&&\quad+\max\{\|\mu\|_{\mathcal{C}},\|\mu\|_{\mathcal{C}}^{q/2}\}\int_{\Delta_{m_4d}}\left[\tilde{N}^{(1+a)d}(\nabla u)\right]^{q}\,dx'\nonumber\\
&&\quad\nonumber
+C\int_{\Delta_{m_4d}\setminus\Delta_{m_2d}}\left[\tilde{N}^{(1+a)d}(\nabla u)\right]^{p}\,dx'\\
&&\quad
+d^{n-1}\sup_{x\in {\mathbb R}^n_+\setminus T(\Delta_{(1+a)d})}[W_2(x)]^q.\nonumber
\end{eqnarray}
Here we truncated $\tilde N$ on the righthand side at the height $(1+a)d$ instead of $m_4d$ since everything above this height
can be incorporated into the last term.

Clearly, for sufficiently small $\|\mu\|_{\mathcal{C}}$ we can hide part of the second term in the last line on the righthand side of \eqref{altsim4a}. Hence
\begin{eqnarray}\label{altsim4}
&&\|\tilde{N}^{(1+a)d}(\nabla u)\|^q_{L^{q}(\Delta_{d})} \lesssim 
\int_{\Delta_{m_4d}}|\nabla_T f(x')|^{q}\,dx' \\
&+&C\|\tilde{N}^{(1+a)d}(\nabla u)\|^q_{L^{q}(\Delta_{m_4d}\setminus\Delta_{d})}
+d^{n-1}\sup_{x\in {\mathbb R}^n_+\setminus T(\Delta_{(1+a)d}))}[W_2(x)]^q.\nonumber
\end{eqnarray}

Clearly, the last estimate is scale invariant and so we write it instead for an enlarged ball $\Delta_{(1+d)d}$. We do this to have in the second term $\Delta_{m_5d}\setminus\Delta_{(1+a)d}$ where $m_5=(1+a)m_4$. Since $\|\tilde{N}^{(1+a)d}(\nabla u)\|_{L^{q}(\Delta_{d})}\le \|\tilde{N}^{(1+a)d}(\nabla u)\|_{L^{q}(\Delta_{(1+a)d})}$ this gives us:
\begin{eqnarray}\label{altsim4a}
&&\|\tilde{N}^{(1+a)d}(\nabla u)\|^q_{L^{q}(\Delta_{d})} \lesssim 
\int_{\Delta_{m_5d}}|\nabla_T f(x')|^{q}\,dx' \\
&+&C\|\tilde{N}^{(1+a)d}(\nabla u)\|^q_{L^{q}(\Delta_{m_5d}\setminus\Delta_{(1+a)d})}
+d^{n-1}\sup_{x\in {\mathbb R}^n_+\setminus T(\Delta_{(1+a)d}))}[W_2(x)]^q.\nonumber
\end{eqnarray}

We now push-forward \eqref{altsim4a} back to the original domain ${\mathcal O}$. We have

\begin{align}\nonumber
\|\tilde{N}^{(1+a)d}_{a/2} (\nabla u)\|^q_{L^{q}(\Delta_d)}&\leq C\|\nabla_T f\|^q_{L^{q}(\partial{\mathcal O\cap \overline{T(\Delta_{m_5d})}})}+Cd^{n-1}\sup_{x\in \mathcal O\cap\{\delta(x)>d\}}W_2(x)^q\\\label{Main-Estlocxx-brb}
&\quad +C\|\tilde{N}^{(1+a)d}(\nabla u)\|^q_{L^{q}({\partial \mathcal O}\cap [T(\Delta_{m_5d})\setminus T(\Delta_{(1+a)d})])}.
\end{align}
We would like to hide the last term. Observe that
all points of ${\partial O}\cap[T(\Delta_{m_5d})\setminus T(\Delta_{(1+a)d})]$ are in the interior of the original domain ${\mathbb R}^n_+$ of distance at least $d$ away from the boundary of ${\mathbb R}^n_+$. Hence whenever we were applying the Theorem \ref{T:Car} we could have in fact used \eqref{Ca-222-x} there with $h$ being the function describing the boundary of $\mathcal O$. Since pointwise for $Q\in \partial \mathcal O\cap [T(\Delta_{m_5d})\setminus T(\Delta_{(1+a)d})]$
$$\tilde{N}_{a,h}(\nabla u)(Q)\le \sup_{x\in \mathcal O\cap\{\delta(x)>d\}}W_2(x)$$
the last term can be estimated by
$Cd^{n-1}\sup_{x\in \mathcal O\cap\{\delta(x)>d\}}W_2(x)^q$ as well.

Finally, we can remove the truncation of $\tilde N$ at height $(1+a)d$ on the lefthhand side of \eqref{Main-Estlocxx-brb} as for points above this height again the term $Cd^{n-1}\sup_{x\in \mathcal O\cap\{\delta(x)>d\}}W_2(x)^q$ controls the nontangential maximal function. This establishes our claim.
\end{proof}

\section{Proof of Theorem \ref{S3:T0}.}\label{S5}

We will establish the solvability of the Regularity problem assuming that the coefficients of $A$ and $B$ are smooth, applying the results of the previous two sections. The constants in the estimates will not depend on the degree of smoothness. Then, considering smooth approximations of ${\mathcal L}$, a limiting argument gives Theorem \ref{S3:T0} in the general case.

We start with $p=2$. Assume the matrix $A$ is $2$-elliptic.
It follows that Lemma \ref{S6:L1} applies. For any $K$ as in the Lemma for any $\|\mu\|_{\mathcal C}<K$ we pick $a$ such that $\|\mu\|_{\mathcal C}+a^{-1}<K$.

Consider any $f\in L^2(\partial\mathbb R^n_+)\cap \dot{B}^{2,2}_{1/2}(\partial\mathbb R^n_+)$ and let $u\in\dot{W}^{1,2}(\mathbb R^n_+)$ be the unique energy solution of $\mathcal Lu=0$ with boundary datum $f$. We shall additionally assume the $f$ is a smooth compactly supported function, it suffices to establish our estimates for those as such functions form a dense subset of $L^2(\partial\mathbb R^n_+)\cap \dot{B}^{2,2}_{1/2}(\partial\mathbb R^n_+)$.

Fix $d>0$  and consider $\Delta=\Delta_d(0)$.  We apply Lemma  \ref{S6:L1}  to the domains 
${\mathcal O}_\tau={\mathcal O}_{\tau\Delta,a}$, for $\tau\in [1,2]$. This gives us
\begin{equation}\label{Main-Estlocxx2}
\|\tilde{N}_{a/2} (\nabla u)\|^2_{L^{2}(\Delta)}\leq C\|\nabla_T f\|^2_{L^{2}(\partial{\mathcal O_\tau\cap \overline{T(\tau m\Delta)}})}+Cd^{n-1}\sup_{x\in {\mathcal O}_\tau \cap\{\delta(x)>d\}}W_2(x)^2.
\end{equation}
Note that each of the sets $\partial{\mathcal O_\tau\cap \overline{T(\tau m\Delta)}}$ consists of the 
\lq\lq flat piece" that is just $\tau\Delta=\Delta_{\tau d}(0)$ and the remaining curve that lies inside $\mathbb R^n_+$. If we average the above inequality over all values of $\tau\in [1,2]$ the latter turns into a solid integral over a set that is contained in
$${\mathcal S}_d:=(0,2md)\times (\Delta_{2md}\setminus \Delta_d).$$
It follows that
\begin{align}\label{Main-Estlocxx2-av}
\|\tilde{N}_{a/2} (\nabla u)\|^2_{L^{2}(\Delta)}\leq& C\|\nabla_T f\|^2_{L^{2}(2\Delta)}+Cd^{n-1}\sup_{\{x:\,\delta(x)>d\}}W_2(x)^2\\
&+Cd^{-1}\iint_{\mathcal S_d}|\nabla u|^2 dx.\nonumber
\end{align}

Consider what happens as we take $d\to\infty$ in the estimate above. Recall that we know that
$\nabla u\in L^2(\mathbb R^n_+)$ from the fact that $u$ is an energy solution. This information implies
that both 
$$\iint_{B(x,\delta(x)/2)}|\nabla u|^2\,dx\to 0,\qquad \iint_{\mathcal S_d}|\nabla u|^2 dx\to 0,$$
for all $x\in \{x:\delta(x)>d\}$ uniformly as $d\to\infty$. From this however we see that the last two terms of \eqref{Main-Estlocxx2-av} go to zero as $d\to\infty$ and hence in the limit we have
$$\|\tilde{N}_{a/2} (\nabla u)\|^2_{L^{2}(\partial\mathbb R^n_+)}\leq C\|\nabla_T f\|^2_{L^{2}(\partial\mathbb R^n_+)},$$
which is $L^2$ solvability of the Regularity problem. Also observe that constant $C$ in the estimate above only depends on $\lambda_2,\, \Lambda$ and  $n$, precisely as stated in Theorem \ref{S3:T0}.

We now extrapolate. It has been established in \cite{DHM}  that, from Lemma \ref{S6:L1}, a purely real variable argument can be used to establish the following estimate
\begin{equation}\label{Main-Estloc5}
\int_{E_\nu\cap\{g\le\nu\}}\left[\tilde{N} (\nabla u)(x')\right]^2\,dx'\le C_\alpha\nu^2|E_\nu| +C\alpha^{-1}\int_{E_\nu}\left[\tilde{N} (\nabla u)(x')\right]^2\,dx',
\end{equation}
where $E_\nu=\{x'\in\mathbb R^{n-1}_+:\,\tilde{N}_\alpha (\nabla u)(x')>\nu \}$ and 
$$g(x')=\sup_{B\ni x'}\left(\dint_{B}|\nabla_Tf(y')|^2dy'\right)^{1/2}.$$
See in particular Lemma 6.1 and (6.17) of \cite{DHM} which are completely analogous to our Lemma \ref{S6:L1} \eqref{Main-Estloc5}. A consequence of \eqref{Main-Estloc5} is an existence of
$\delta>0$ which only depends on the constant in the estimate \eqref{Main-Estlocxx} such that
\begin{equation}\label{goodgrief}
\|\tilde{N}_{a/2} (\nabla u)\|_{L^{2+\delta}(\partial\mathbb R^n_+)}\leq C\|\nabla_T f\|_{L^{2+\delta}(\partial\mathbb R^n_+)},
\end{equation}
which is the solvability of the Regularity problem for $p_0=2+\delta$. 
If $p_0$ is such that the matrix $A$ is $p_0$-elliptic we can repeat the process above we did for $p=2$. We now apply Lemma  \ref{S6:L1} for the value $p_0$ and again take the limit $d\to\infty$.
This time the solid integrals we get are
$$d^{-1}\iint_{B(x,\delta(x)/2)}|\nabla u|^{p_0}\,dx,\qquad d^{-1}\iint_{\mathcal S_d}|\nabla u|^{p_0} dx,$$
which we know go to zero uniformly for all $x\in \{x:\delta(x)>d\}$ as $d\to\infty$ thanks to the fact that
\eqref{goodgrief} implies that $\|\tilde{N}_{p_0,a/2} (\nabla u)\|_{L^{p_0}}<\infty$. Hence taking the limit $d\to\infty$ in the analogue of \eqref{Main-Estlocxx2-av} for $p_0$ yields
\begin{equation}\label{goodgrief2}
\|\tilde{N}_{a/2} (\nabla u)\|_{L^{p_0}(\partial\mathbb R^n_+)}\leq C_{p_0}\|\nabla_T f\|_{L^{p_0}(\partial\mathbb R^n_+)}.
\end{equation}
This seemingly is just a restatement of \eqref{goodgrief}. The difference however is that now the constant $C_{p_0}$ in \eqref{goodgrief2} only depends on the constant in  Lemma \ref{S6:L1} for the value $p_0$. This allows us to extrapolate again and obtain solvability of the Regularity problem for some value $p_0+\delta'$. There is no 
difference in the structure of the argument. We can continue this bootstrapping as long as we stay in the range of $p$-ellipticity and as long as we can
be sure that we are moving by an amount $\delta'$ which is not getting smaller at each step. This last point is assured by the fact that the constants $C_{p_0}$ in the 
$L^{p_0}$ norm inequalities (\ref{goodgrief2}) only depend on the $p_0$-ellipticity and the Carleson measure norm of the coefficients. If we fix $p>2$ such that the operator
is $p$-elliptic the constants $C_q$ for $2\le q\le p$  in Lemma \ref{S6:L1} are uniformly bounded
which assures that our bootstrapping argument will reach the desired value $p$ is finitely many steps giving us solvability of the Regularity problem and the estimate \eqref{Main-Est} of Theorem \ref{S3:T0}.

We now deal with $p<2$ such that $A$ is $p$-elliptic. Assume first that we a priori know that $\|\tilde{N}_{2,a}(\nabla u)\|_{L^p(\mathbb R^{n-1})}<\infty$ for an energy solution $u$ in ${\mathbb R^n_+}$ with boundary datum $f$.
Then by \eqref{S3:C7:E00oo} of Proposition \ref{S3:C7} and by \eqref{S3:C7:E00oo=s2} of Corollary \ref{S4:C2} we have
\begin{align}
\|\tilde{N}_{2,a}(\nabla u)\|_{L^p(\mathbb R^{n-1})}&\le C\|S_{2,a}(\nabla u)\|_{L^p(\mathbb R^{n-1})}\\
&\le C\|S_{2,a}(\nabla_T u)\|_{L^p(\mathbb R^{n-1})}+C\|\mu\|^{1/2}\|\tilde{N}_{2,a}(\nabla u)\|_{L^p(\mathbb R^{n-1})}.\nonumber
\end{align}
Here in order to use Corollary \ref{S4:C2}, we must verify that $\|\mathcal{T}_{a}(\nabla u)\|_{L^p(\mathbb R^{n-1})}<\infty$. However under the assumption that the coefficients are smooth we have a pointwise bound $\mathcal{T}_{a}(\nabla u)(Q)\le S_{2,a}(u)(Q)$. We have established solvability of the $L^p$ Dirichlet problem in the paper \cite{DPcplx} in the range where $p$-ellipticity holds and in particular we have shown the bound $\|S_{2,a}(u)\|_{L^p(\mathbb R^{n-1})}\lesssim\|f\|_{L^p(\mathbb R^{n-1})}<\infty$ (using that $f\in C_0^\infty\subset L^p$).

Hence taking sufficiently small $K$ in Theorem \ref{S3:T0} it follows that
\begin{align}
\|\tilde{N}_{2,a}(\nabla u)\|_{L^p(\mathbb R^{n-1})}&\le  C\|S_{2,a}(\nabla_T u)\|_{L^p(\mathbb R^{n-1})}.\label{eq-77}
\end{align} 
Hence by \eqref{eq-qq'} in conjunction with \eqref{S3:C7:E00oo=s} of Corollary \ref{S4:C1} and \eqref{NN} implies that
\begin{align}\label{zertex4}
\|S_{2,a}(\nabla_T u)\|_{L^p(\mathbb R^{n-1})}\le C\|S_{p,a}(\nabla_T u)\|_{L^p(\mathbb R^{n-1})}+\varepsilon\|\tilde{N}_{2,a}(\nabla u)\|_{L^p(\mathbb R^{n-1})}.
\end{align} 
When applying  \eqref{S3:C7:E00oo=s} to estimate $\|S_{p',a}(\nabla_T u)\|_{L^p(\mathbb R^{n-1})}$
we need to know a priori that this quantity is finite. Here we use our assumption that for now the coefficients are smooth. This gives is a point-wise bound
$$S_{p',a}(\nabla_T u)\le CS_{2,a}(\nabla_T u)+CN(\nabla u),$$
where $N$ is the pointwise maximal function. The classical $L^\infty$ bounds of Agmon-Douglis-Nirenberg \cite{ADN} for smooth PDE systems  imply $N\lesssim N_2$. We also have 
$\|S_{2,a}(\nabla_T u)\|_{L^p}<\infty$ from a similar estimate
$$S_{2,a}(\nabla_T u)\le CS_{p,a}(\nabla_T u)+CN(\nabla u),$$
and finally we know that $\|S_{p,a}(\nabla_T u)\|_{L^p(\mathbb R^{n-1})}<\infty$ by \eqref{square02aa}
(taking $r\to\infty$). The one \lq\lq bad" term in \eqref{square02aa} which is $\int_{\mathbb R^{n-1}}|\nabla_T u(r,x')|^pdx'$ can be dealt with by averaging in $r$ first which turns it into a solid integral. Such term can be estimated by  
$\|\tilde{N}_{p,a}(\nabla u)\|_{L^p(\mathbb R^{n-1})}\lesssim \|\tilde{N}_{2,a}(\nabla u)\|_{L^p(\mathbb R^{n-1})}<\infty$ and furthermore it follows this term converges to zero as $r\to\infty$.

Hence all quantities appearing in \eqref{zertex4} are finite under the assumption our coefficients are smooth, but the constants in this estimate only depend on the parameters $n,p,\lambda_p,\Lambda$.
We choose $\varepsilon>0$ in this inequality  small enough so that we can hide this term on the lefthand side of 
\eqref{eq-77}. This gives is
\begin{align}
\|\tilde{N}_{2,a}(\nabla u)\|_{L^p(\mathbb R^{n-1})}&\le  C\|S_{p,a}(\nabla_T u)\|_{L^p(\mathbb R^{n-1})}.\label{eq-78}
\end{align} 
We can now use again \eqref{square02aa} for $S_{p,a}(\nabla_T u)$ taking $r\to\infty$. As explained above  the term $\int_{\mathbb R^{n-1}}|\nabla_T u(r,x')|^pdx'$  gets eliminated. It follows that  \eqref{square02aa} gives us
\begin{align}
\|S_{p,a}(\nabla_T u)\|_{L^p(\mathbb R^{n-1})}&\le C\|\nabla_T f\|_{L^p(\mathbb R^{n-1})}+C\|\mu\|_{\mathcal C}^{1/p}\|\tilde{N}_{p,a}(\nabla u)\|_{L^p(\mathbb R^{n-1})}.\label{eq-79}
\end{align} 
Hence for all $\|\mu\|_{\mathcal C}<K$ sufficiently small combination of \eqref{eq-78}, \eqref{eq-79}
and \eqref{NN} yields
\begin{align}\label{goodgrief3}
\|\tilde{N}_{2,a}(\nabla u)\|_{L^p(\mathbb R^{n-1})}&\le C\|\nabla_T f\|_{L^p(\mathbb R^{n-1})},
\end{align}
from which solvability of the $L^p$ Regularity problem follows. \vglue1mm

It remain to remove the a priori assumption $\|\tilde{N}_{2,a}(\nabla u)\|_{L^p(\mathbb R^{n-1})}<\infty$ we have made earlier.

We again argue by extrapolation starting with $p=2$ where we know this since we have already established solvability of the Regularity problem for this value of $p$.

This time we shall use an extrapolation argument based on an method in \cite{DKV} of obtaining $L^{2-\varepsilon}$ estimates of nontangential maximal functions
from $L^2$ estimates on sawtooth domains. See also \cite{DHM}, where this technique was used to get solvability of the $L^p$ Dirichlet problem for elliptic systems for
$2 - \varepsilon < p < 2$.
In particular, the argument of \cite{DKV}, reproduced in section 6 of \cite{DHM} for systems and hence valid in our setting, gives that $\|\tilde{N}_{2,a}(\nabla u)\|_{L^{p_0}({\BBR}^{n-1})}< \infty$
for $p_0=2 - \varepsilon$ and hence the same is true for $\|\tilde{N}_{p_0,a}(\nabla u)\|_{L^{p_0}({\BBR}^{n-1})}$. The quantity $\varepsilon$ depends on the constant $C_2$ in the
$L^2$ norm inequality between the nontangential maximal function and the square function $S_2$.  

Once we know these quantities are finite, the calculation we did above holds for $p_0$ giving us 
(\ref{goodgrief3}), and hence the same estimate for $\nabla u$, for $p_0=2-\varepsilon$ 
and a constant $C_{2-\epsilon}$.

The very same extrapolation argument, now invoking the $L^{p_0}$ estimate gives an $L^{p_0 - \varepsilon'}$ estimate where $\varepsilon'$ now depends on $C_{2-\varepsilon}$. In other words,
we apply the same argument as \cite{DKV} but starting from known estimates for the nontangential maximal function in  $L^{p_0}$ instead of $L^2$. We can continue this bootstrapping as long as we stay in the range of $p$-ellipticity and as long as we can
be sure that we are moving by an amount $\varepsilon$ which is not getting smaller at each step. The same argument as given previously implies that we can reach any value $p<2$ in the $p$-ellipticity range of the matrix $A$ in finite number of steps. From this Theorems \ref{S3:T0} follows.\vglue1mm

Finally, we remove the temporary assumption that the coefficients are smooth. The key is that the constants in the estimates above depend only on $n,p,\lambda_p,\Lambda,\|\mu\|_{\mathcal C}$ and not on any further degree of smoothness of the coefficients of $\mathcal L$. Hence the classical argument where we approximate our coefficients by smooth functions, and then pass from the smooth coefficient case by taking the limit can be applied. See for example section 4 of \cite{DPcplx} where this is discussed in more detail.
\qed\vglue1mm

\section{Proof of Theorem \ref{S3:T1}.}\label{S6}

The proof is based on the following abstract result \cite{Sh1}, see also \cite[Theorem 3.1]{WZ} for a version on an arbitrary bounded domain.

\begin{theorem}\label{th-sh} Let $T$ be a bounded sublinear operator on $L^2({\mathbb R}^{n-1};{\mathbb C}^m)$. Suppose that
for some $p>2$, $T$ satisfies the following $L^p$ localization property. For any ball $\Delta=\Delta_d\subset{\mathbb R}^{n-1}$ and $C^\infty$ function $f$ with supp$(f)\subset{\mathbb R}^{n-1}\setminus 3\Delta$ the following estimate holds:
\begin{align}
&\left(|\Delta|^{-1}\int_\Delta|Tf|^p\,dx'\right)^{1/p}\le\label{eq-pf8}\\
&\qquad C\left\{\left(|2\Delta|^{-1}\int_{2\Delta}|Tf|^2\,dx'\right)^{1/2}+\sup_{\Delta'\supset \Delta}\left(|\Delta'|^{-1}\int_{\Delta'}|f|^2\,dx'\right)^{1/2}  \right\},\nonumber
\end{align}
for some $C>0$ independent of $f$. Then $T$ is bounded on $L^q({\mathbb R}^{n-1};{\mathbb C}^m)$ for any $2\le q<p$.
\end{theorem}
In our case the role of $T$ is played by the sublinear operator $f\mapsto \tilde{N}_{2,a}(u)$, where $u$ is the solution of the Dirichlet problem ${\mathcal L}u=0$ with boundary data $f$. Clearly, in the Theorem above
the factors $2\Delta$, $3\Delta$ do not play significant role. Hence if we establish estimate \eqref{eq-pf8}
with $2\Delta$ replaced by $m\Delta$ with $f$ vanishing on $(m+1)\Delta$ for some $m>1$ the claim of the Theorem will remain to hold.

Clearly, our operator $T:f\mapsto \tilde{N}_{2,a}(u)$ is sublinear and bounded on $L^2$
by \cite{DPcplx}, for coefficients with small Carleson norm $\mu$. To prove \eqref{eq-pf8} we shall establish the following reverse H\"older inequality, following ideas of Shen \cite{Sh2}.
\begin{equation}
\left(\frac1{|\Delta|}\int_\Delta|\tilde{N}_{2,a}(u)|^p\,dx'\right)^{1/p}\le\label{eq-pf9}
C\left(\frac1{|3\beta m\Delta|}\int_{3\beta m\Delta}|\tilde{N}_{2,a}(u)|^2\,dx'\right)^{1/2},
\end{equation}
where $f=u\big|_{\partial{\mathbb R}^n_+}$ vanishes on $4\beta m\Delta$. Here $m$ is determined by Lemma \ref{S6:L1} and $\beta>1$ is determined by a bootstrap argument explained later.
Having this by Theorem \ref{th-sh} we have for any $q \in [2,p)$ the estimate
\begin{equation}\label{eq-pf10}
\|\tilde{N}_{2,a}(u)\|_{L^{q}({\BBR}^{n-1})}\le C\|f\|_{L^{q}({\BBR}^{n-1})},
\end{equation}
which implies $L^q$ solvability of the Dirichlet problem for the operator $\mathcal L$.

It remains to establish \eqref{eq-pf9}. Let us define
\begin{align}\label{eq-pf11}
&{\mathcal M}_1(u)(x')=\sup_{y\in\Gamma_a(x')}\{w_2(y):\,\delta(y)\le cd\},\\
&{\mathcal M}_2(u)(x')=\sup_{y\in\Gamma_a(x')}\{w_2(y):\,\delta(y)> cd\}.\nonumber
\end{align}
where $c=c(a)>0$ is chosen such that for all $x'\in\Delta$ if $y=(y_0,y')\in\Gamma_a(x')$ and $y_0=\delta(y)\le cd$ then $y'\in 2\Delta$. Here $d=\mbox{diam}(\Delta)$ and $w_2$ is the $L^2$ average of $u$
$$w_2(y)=\left(\dint_{B_{\delta(y)/2}(y)}|u(z)|^2\,dz\right)^{1/2}.$$
It follows that
$$\tilde{N}_{2,a}(u)=\max\{{\mathcal M}_1(u),{\mathcal M}_2(u)\}.$$
We first estimate ${\mathcal M}_2(u)$. Pick any $x'\in\Delta$.
For any $y\in\Gamma(x')$ with $\delta(y)>cd$ it follows that
for a large subset $A$ of $2\Delta$ (of size comparable to $2\Delta$) we have
$$z'\in A\quad\Longrightarrow\quad y\in\Gamma_a(z')\quad\Longrightarrow\quad w_2(y)\le \tilde{N}_{2,a}(u)(z').$$
Hence for any $x'\in\Delta$ 
$${\mathcal M}_2(u)(x')\le C\left(\frac1{|2\Delta|}\int_{2\Delta} \left[\tilde{N}_{2,a}(u)(z')\right]^2\,dz'\right)^{1/2}.$$
It remains to estimate ${\mathcal M}_1(u)$ on $\Delta$.

We write
$$u(x_0,x')-u(0,y')=\int_{0}^{1}\frac{\partial u}{\partial s}(sx_0,(1-s)y'+sx')\,ds.$$
Let $K=\{(y_0,y'):y'\in \Delta\mbox{ and }cd<y_0<2cd\}$. Using the previous line and the fact that $u$ vanishes on $3\Delta\subset 4\beta m\Delta$ we have for any $x'\in\Delta$
\begin{equation}\label{eq-pf12}
{\mathcal M}_1(u)(x')\le \sup_{K}w_2\,+\,C\int_{2\Delta}\frac{\tilde{N}_{2,a/2}(\nabla u)(y')}{|x'-y'|^{n-2}}dy'.
\end{equation}
By the fractional integral estimate, this implies that
\begin{equation}\label{eq-pf13}
\left(\frac1{|\Delta|}\int_\Delta[{\mathcal M}_1(u)(x')]^p\,dx'\right)^{1/p}\le \sup_{K}w_2\,+\,Cd
\left(\frac1{|2\Delta|}\int_{2\Delta}[\tilde{N}_{2,a/2}(\nabla u)(x')]^q\,dx'\right)^{1/q},
\end{equation}
where $\frac1p=\frac1q-\frac1{n-1}$ and $1<q<n-1$.\vglue1mm

To further estimate \eqref{eq-pf13} we use the Lemma \ref{S6:L1}. We claim the following reverse H\"older inequality holds
$$\left(\frac1{|\Delta|}\int_{\Delta}[\tilde{N}_{2,a/2}(\nabla u)(x')]^q\,dx'\right)^{1/q}\lesssim \left(\frac1{|\beta \Delta|}\int_{\beta \Delta}[\tilde{N}_{2,a/2}(\nabla u)(x')]^2\,dx'\right)^{1/2},$$
whenever the solution $\mathcal Lu=0$ vanishes on at least $2\beta \Delta$.

Let $d$ be the diameter of $\Delta$. We apply Lemma \ref{S6:L1} to the domains (\ref{Odom})
${\mathcal O}_\tau={\mathcal O}_{\tau\Delta,a}$, for $\tau\in [1,2]$. This gives us
\begin{equation}\label{Main-Estlocxx2x}
\|\tilde{N}_{a/2} (\nabla u)\|_{L^{q}(\Delta)}\leq C\|\nabla_T f\|_{L^{q}(\partial{\mathcal O_\tau\cap \overline{T(\tau m\Delta)}})}+Cd^{(n-1)/q}\sup_{x\in {\mathcal O}_\tau \cap\{\delta(x)>d\}}W_2(x).
\end{equation}
Observe that for any $x\in {\mathcal O}_\tau \cap\{\delta(x)>d\}$ we shall have
$$|A|=|\{y'\in 2\Delta:x\in\Gamma_{a/2}(y')\}|\approx d^{n-1},$$
and clearly for each $y'\in A$ we have $W_2(x)\lesssim \tilde{N}_{a/2} (\nabla u)(y')$, from which
$$W_2(x)\lesssim |A|^{-1}\left(\int_A[\tilde{N}_{a/2} (\nabla u)(y')]^2dy'\right)^{1/2}\lesssim |2\Delta|^{-1}\left(\int_{2\Delta}[\tilde{N}_{a/2} (\nabla u)(y')]^2dy'\right)^{1/2}.$$
It follows 
\begin{equation}\label{eq-add1}
\sup_{x\in \mathcal O_\tau \cap\{\delta(x)>d\}}W_2(x)\lesssim 
 |2\Delta|^{-1}\left(\int_{2\Delta}[\tilde{N}_{a/2} (\nabla u)(y')]^2dy'\right)^{1/2}.
 \end{equation}
 We use this in \eqref{Main-Estlocxx2x}, integrate \eqref{Main-Estlocxx2x} in $\tau$ over the interval $[1,2]$ and divide by $d^{(n-1)/q}$. This gives after using the fact that $u=0$ vanishes on at least $4m\Delta$:
\begin{align}\label{eq-pf17}
&\left(\frac1{|\Delta|}\int_{\Delta}[\tilde{N}_{2,a/2}(\nabla u)(x')]^q\,dx'\right)^{1/q}\\
\lesssim &\quad
\left(\frac1{T(2m\Delta)}\iint_{T(2m\Delta)}|\nabla u(x)|^q\,dx\right)^{1/q} +\left(\frac1{|2\Delta|}\int_{2\Delta} \left[\tilde{N}_{2,a}(\nabla u)(x')\right]^2\,dx'\right)^{1/2}.\nonumber
\end{align}
We have also used the trivial estimate $|\nabla_Tu|\le|\nabla u|$ on $\partial {\mathcal O}_\tau\cap{T(2m\Delta)}$.  For the first term we have

\begin{align}\label{eq-add2}
\iint_{T(2m\Delta)}|\nabla u(x)|^q\,dx&=\iint_{T(2m\Delta)\cap\{x_0<\varepsilon md\}}|\nabla u(x)|^q\,dx\\\nonumber &+\iint_{T(2m\Delta)\cap\{x_0> \varepsilon md\}}|\nabla u(x)|^q\,dx.
\end{align}
The set $T(2m\Delta)\cap\{x_0> \varepsilon md\}$ in the the interior of ${\mathbb R}^n_+$ of diameter and distance to the boundary that is comparable to $d$. It follows that the interior estimate \eqref{RHthm1} can be used (we only enlarge this set by a small factor $\alpha>1$ so that $\alpha[T(2m\Delta)\cap\{x_0> \varepsilon md\}]$ fully lies in the interior of ${\mathbb R}^n_+$. It follows
\begin{align}\label{eq-add3}
&\quad\frac1{|T(2m\Delta)|}\iint_{T(2m\Delta)\cap\{x_0> \varepsilon md\}}|\nabla u(x)|^q\,dx  \\&\lesssim\left(\frac1{|T(2m\Delta)|}\iint_{\alpha[T(2m\Delta)\cap\{x_0> \varepsilon md\}]}|\nabla u(x)|^2\,dx\right)^{q/2}\nonumber\\\nonumber
&\lesssim \left(\frac1{|T(3m\Delta)|}\iint_{T(3m\Delta)}|\nabla u(x)|^2\,dx\right)^{q/2}\\\nonumber
&\lesssim  \left(\frac1{|3m\Delta|}\int_{3m\Delta}[\tilde{N}_{a/2} (\nabla u)(x')]^2dx'\right)^{q/2}.
\end{align}
For the term $\iint_{T(2m\Delta)\cap\{x_0<\varepsilon md\}}|\nabla u(x)|^q\,dx$
we use the trivial estimate
\begin{align}\label{eq-add4}
\iint_{T(2m\Delta)\cap\{x_0<\varepsilon md\}}|\nabla u(x)|^q\,dx\le \varepsilon md 
\int_{3m\Delta}[\tilde{N}_{a/2} (\nabla u)(x')]^qdx'.
\end{align}
Combining \eqref{eq-add2}-\eqref{eq-add4} finally yields
\begin{align}\label{eq-add5}
&\quad\frac1{|T(2m\Delta)|}\iint_{T(2m\Delta)}|\nabla u(x)|^q\,dx\\
&\le \left(\frac{C_\varepsilon}{|3m\Delta|}\int_{3m\Delta}[\tilde{N}_{a/2} (\nabla u)(x')]^2dx'\right)^{q/2}
+\frac\varepsilon{|3m\Delta|}\int_{3m\Delta}[\tilde{N}_{a/2} (\nabla u)(x')]^qdx'.\nonumber
\end{align}
This combined with \eqref{eq-pf17} yields:
\begin{align}\label{eq-add6}
&\frac1{|\Delta|}\int_{\Delta}[\tilde{N}_{2,a/2}(\nabla u)(x')]^q\,dx'\\
\lesssim &\quad
\left(\frac{C_\varepsilon}{|3m\Delta|}\int_{3m\Delta}[\tilde{N}_{a/2} (\nabla u)(x')]^2dx'\right)^{q/2}
+\frac\varepsilon{|3m\Delta|}\int_{3m\Delta}[\tilde{N}_{a/2} (\nabla u)(x')]^qdx'.\nonumber
\end{align}

We now recall an abstract result from \cite[Chapter 5; Proposition 1.1]{Gi}.

\begin{theorem}\label{thm-gi} Let $B_R$ be a ball in ${\mathbb R}^N$. Suppose that $g\ge 0$, $g\in L^q(B_R)$ for some $q>1$
and for all $x\in B_{R/2}$ and $0<r<R/16$ we have
$$\dint_{B_r} g^q\,dx \le C\left(\dint_{B_{2r}} g\,dx\right)^q +\theta \dint_{B_{2r}}g^q\, dx,$$
for some constants $C>1$, $\theta<1$.

Then there exists $\delta=\delta(C,\theta,N,q)>0$ and $K=K(C,\theta,N,q)>0$ such that for all $B_r$
concentric with $B_R$ of radius $0<r<R/4$ we have
$$\left(\dint_{B_{r/2}} g^{q+\delta}\,dx\right)^{1/(q+\delta)} \le K\left(\dint_{B_{r}} g^q\,dx\right)^{1/q} .$$
\end{theorem}
Applying this to \eqref{eq-add6} with $g(x')=[\tilde{N}_{a/2} (\nabla u)(x')]^2$ 
yields that for some $\alpha>1$ we have
\begin{align}\label{eq-add7}
&\left(\frac1{|\Delta|}\int_{\Delta}[\tilde{N}_{2,a/2}(\nabla u)(x')]^{q+\delta}\,dx'\right)^{1/(q+\delta)}
\lesssim
\left(\frac{1}{|\alpha\Delta|}\int_{\alpha\Delta}[\tilde{N}_{a/2} (\nabla u)(x')]^q dx'\right)^{1/q}.
\end{align}
Here clearly, $\delta=\delta(q)$ depends on $q$ but as long as the constant $C_\varepsilon$ in the estimate 
\eqref{eq-add5} stays uniform (which is for $q\in [p_0+\eta,p_0'-\eta]$ for any $\eta>0$ where $(p_0,p_0')$ is the interval we have $p$-ellipticity) we shall have
$$\inf_{q\in [2+\eta,p_0'-\eta]}\delta(q)>0,\qquad\mbox{for all }\eta>0.$$
Here we are avoiding $q$ near $2$ as well since then \eqref{eq-add5} provides no information. However, to get us started in the bootstrap argument we may use the inequality
$$\left(\frac1{T(2m\Delta)}\iint_{T(2m\Delta)}|\nabla u|^{2+\delta_0}\,dx\right)^{1/(2+\delta_0)}\lesssim \left(\frac1{T(3m\Delta)}\iint_{T(3m\Delta)}|\nabla u|^2\,dx\right)^{1/2},$$
for some $\delta_0>0$ small 
which is a well known consequence of the Caccioppoli's inequality and Theorem \ref{thm-gi}.
It follows using \eqref{eq-pf17}
\begin{align}\nonumber
&\left(\frac1{|\Delta|}\int_{\Delta}[\tilde{N}_{2,a/2}(\nabla u)(x')]^{2+\delta_0}\,dx'\right)^{1/(2+\delta_0)}\lesssim
\left(\frac1{|3m\Delta|}\int_{3m\Delta} \left[\tilde{N}_{2,a}(\nabla u)(x')\right]^2\,dx'\right)^{1/2}.
\end{align}
This is the initial inequality in the bootstrap argument after which we iteratively use \eqref{eq-add7} where
$\delta>0$ stays bounded away from zero as long as we take $q\le p_0'-\eta$ for some small fixed $\eta>0$.
This finally implies that for all $q<p_0'$ we have
\begin{align}\label{eq-add8}
&\left(\frac1{|\Delta|}\int_{\Delta}[\tilde{N}_{2,a/2}(\nabla u)(x')]^{q}\,dx'\right)^{1/q}\lesssim
\left(\frac1{|\beta\Delta|}\int_{\beta \Delta} \left[\tilde{N}_{2,a}(\nabla u)(x')\right]^2\,dx'\right)^{1/2},
\end{align}
for some $\beta>1$ with $u$ vanishing on $2\beta\Delta$. The implied constant in the estimate \eqref{eq-add8}
gets progressively worse as $q\to p_0'-$. Next, we use again \eqref{eq-pf17} but this time for $q=2$

\begin{align}\label{eq-add9}
&\left(\frac1{|\beta\Delta|}\int_{\beta\Delta}[\tilde{N}_{2,a/2}(\nabla u)(x')]^2\,dx'\right)^{1/2}\\
\lesssim &\quad \left(\frac1{T(2\beta m\Delta)}\iint_{T(2\beta m\Delta)}|\nabla u|^2\,dx\right)^{1/2}
+\sup_{x\in \mathcal O_{2\beta}\cap\{\delta(x)>d\}}W_2(x),\nonumber
\end{align}
where we put back $W_2$ instead of our initial estimate \eqref{eq-add1}. For the first term we use the boundary Caccioppoli's inequality 
\begin{align}\nonumber
\left(\frac1{T(2\beta m\Delta)}\iint_{T(2\beta m\Delta)}|\nabla u|^2\,dx\right)^{1/2}&\lesssim d^{-1}\left(\frac1{T(3\beta m\Delta)}\iint_{T(3\beta m\Delta)}|u|^2\,dx\right)^{1/2}\\\nonumber
&\lesssim d^{-1}\left(\frac1{|3\beta m\Delta|}\int_{3\beta m\Delta} \left[\tilde{N}_{2,a}(u)(z')\right]^2\,dz'\right)^{1/2},\nonumber
\end{align}
while for the second term by the interior Ciacciopoli's inequality we have for all $x\in{\mathbb R}^n_+$ with $\delta(x)>d$
$$W_2(x)\le Cd^{-1}w_2(x),$$
where $w_2$ denotes the $L^2$ averages of $u$ (defined earlier). We have intentionally shrunk the size of the ball in the definition of $W_2$ so that this pointwise estimate holds. Since the $x$ we consider in the supremum is in $\mathcal O_{2\beta}$ it then follows
\begin{equation}\label{eq-pf16}
\sup_{x\in {\mathcal O_{2\beta}}\cap\{\delta(x)>d\}}W_2(x)\lesssim d^{-1}\left(\frac1{|2\beta\Delta|}\int_{2\beta\Delta} \left[\tilde{N}_{2,a}(u)(z')\right]^2\,dz'\right)^{1/2}.
\end{equation}
Using this and the previous estimates \eqref{eq-add8}-\eqref{eq-add9} then yield for all $q<p_0'$
\begin{align}\label{eq-add10}
&\left(\frac1{|\Delta|}\int_{\Delta}[\tilde{N}_{2,a/2}(\nabla u)(x')]^{q}\,dx'\right)^{1/q}\lesssim
d^{-1}\left(\frac1{|3\beta m\Delta|}\int_{3\beta m\Delta} \left[\tilde{N}_{2,a}(u)(z')\right]^2\,dz'\right)^{1/2}.
\end{align}

Finally, inserting this estimate into \eqref{eq-pf13} yields
\begin{equation}\label{eq-pf19}
\left(\frac1{|\Delta|}\int_\Delta[{\mathcal M}_1(u)(x')]^p\,dx'\right)^{1/p}\le C\left(\frac1{|3\beta m\Delta|}\int_{3\beta m\Delta} \left[\tilde{N}_{2,a}(u)(z')\right]^2\,dz'\right)^{1/2},
\end{equation}
where $\frac1p=\frac1q-\frac1{n-1}$ and $1<q<n-1$ such that $A$ is $q$-elliptic and Carleson norm of $\mu$ is small. Since we have assumed $A$ is $q$-elliptic for $q\in (p_0,p_0')$ and $p_0'>2$ this implies in dimensions $2$ and $3$ that we can consider any $2<p<\infty$, while in dimensions $n\ge 4$ we can have
$2<p<p_{\max}=p_0'(n-1)/(n-1-p_0')$ when $p_0'<n-1$, $p_{\max}=\infty$ otherwise. Observe that always $p_{max}>2(n-1)/(n-3)$. From this claim of Theorem \ref{S3:T1} follows as we have established \eqref{eq-pf9} for such values of $p$.
\vglue1mm

\end{document}